\documentclass[letterpaper,11pt,reqno,tbtags]{amsart} 
\usepackage[portrait,margin=3cm]{geometry}  
\usepackage{mathrsfs}
\usepackage[
    colorlinks=true,%
    breaklinks,
    hyperindex,%
    plainpages=false,%
    bookmarksopen=false,%
    linkcolor=blue,%
    anchorcolor=blue,%
    citecolor=blue,%
    filecolor=black,%
    menucolor=black,%
    pagecolor=black,%
    urlcolor=blue,%
    pdfview=FitH,
    pdfstartview=FitH,
    hyperfigures,
    bookmarksnumbered%
  ]{hyperref} 
\usepackage{amsmath,amssymb,amsthm,amsfonts,amsbsy,latexsym,dsfont,tikz}

\usepackage{enumerate}

\newcommand{\mbf}[1]{\mbox{\boldmath$#1$}}

\newcommand{\diag}[0]{\text{diag}}

\newcommand{\R}[0]{\mathbb{R}}
\newcommand{\Id}[0]{\text{Id}}

\newcommand{\valpha}[0]{\mbf \alpha}
\newcommand{\vepsilon}[0]{\mbf \varepsilon}
\newcommand{\vmu}[0]{\mbf \mu}

\newcommand{\Prob}[0]{\mathds{P}}

\newcommand{\SNR}{\mathsf{SNR}}

\newcommand{\op}[0]{\text{op}}

\theoremstyle{remark}
\newtheorem{thm}{Theorem}[section]
\newtheorem{lem}[thm]{Lemma}
\newtheorem{cor}[thm]{Corollary}

\newtheorem{rem}[thm]{Remark}

\renewcommand{\leq}{\leqslant} 
\renewcommand{\geq}{\geqslant}

\def\qed{ \hfill $\blacksquare$}  
\newcommand{\rd}{\mathrm{d}}
\newcommand{\re}{\mathrm{e}}

\newcommand{\cA}{\mathcal{A}}\newcommand{\cB}{\mathcal{B}}

\newcommand{\cL}{\mathcal{L}}

\newcommand{\vzero}{\mathbf{0}}\newcommand{\vone}{\mathbf{1}}

\newcommand{\vX}{\mathbf{X}}\newcommand{\vZ}{\mathbf{Z}}
\newcommand{\vc}{\mathbf{c}}

\newcommand{\vr}{\mathbf{r}}
\newcommand{\vu}{\mathbf{u}}
\newcommand{\vv}{\mathbf{v}}
 
\newcommand{\mv}[1]{\boldsymbol{#1}}

\newcommand{\mvtheta}{\boldsymbol{\theta}}

\newcommand{\mvmu}{\boldsymbol{\mu}}

\newcommand{\bE}{\mathbb{E}}

\newcommand{\bR}{\mathbb{R}}

\newcommand{\bV}{\mathbb{V}}
        
\newcommand{\sC}{\mathscr{C}}

\DeclareMathOperator{\E}{\mathds{E}}

\DeclareMathOperator{\argmax}{argmax}

\DeclareMathOperator{\tr}{tr} 
  
\begin{document}

\title[Sharp threshold of Gaussian mixtures]{Cutoff for exact recovery of Gaussian mixture models}
\author[Xiaohui Chen]{Xiaohui Chen}
\address{\newline Xiaohui Chen\newline
Department of Statistics\newline
University of Illinois at Urbana-Champaign\newline
725 S. Wright Street, Champaign, IL 61820\newline
{\it E-mail}: \href{mailto:xhchen@illinois.edu}{\tt xhchen@illinois.edu}\newline 
{\it URL}: \href{http://publish.illinois.edu/xiaohuichen/}{\tt http://publish.illinois.edu/xiaohuichen/}
}
\author[Yun Yang]{Yun Yang}
\address{\newline Yun Yang\newline
Department of Statistics\newline
University of Illinois at Urbana-Champaign\newline
725 S. Wright Street, Champaign, IL 61820\newline
{\it E-mail}: \href{mailto:yy84@illinois.edu}{\tt yy84@illinois.edu}\newline 
{\it URL}: \href{https://sites.google.com/site/yunyangstat/}{\tt https://sites.google.com/site/yunyangstat/}
}

\date{First arXiv version: January 5, 2020. This version: \today}

\begin{abstract}
We determine the information-theoretic cutoff value on separation of cluster centers for exact recovery of cluster labels in a $K$-component Gaussian mixture model with equal cluster sizes. Moreover, we show that a semidefinite programming (SDP) relaxation of the $K$-means clustering method achieves such sharp threshold for exact recovery without assuming the symmetry of cluster centers.
\end{abstract}
\maketitle


\section{Introduction}
\label{sec:introduction}
Let $\vX_{1},\dots,\vX_{n}$ be a sequence of independent random vectors in $\bR^{p}$ sampled from a $K$-component Gaussian mixture model with $K \leq n$. Specifically, we assume that there exists a partition $G_{1}^{*},\dots,G_{K}^{*}$ of the index set $[n] := \{1,\dots,n\}$ such that if $i \in G_{k}^{*}$, then 
\begin{equation}
\label{eqn:clustering_model}
\vX_{i} = \mvmu_{k} + \vepsilon_i, \quad \vepsilon_i \overset{\mbox{i.i.d.}}{\sim} N(\mathbf 0, \sigma^{2} I_{p}),
\end{equation}
where $\mvmu_{1},\dots,\mvmu_{K} \in \bR^{p}$ are the unknown cluster centers and $\sigma^{2} > 0$ is the common noise variance. For simplicity, we assume that $\sigma^{2}$ is known. Our main focus of this paper is to investigate the problem of optimal exact recovery for the true partition (or clustering) structure $G_{1}^{*},\dots,G_{K}^{*}$.

For each partition $G_{1},\dots,G_{K}$ of $[n]$, let $H = (h_{ik}) \in \{0, 1\}^{n \times K}$ be the binary {\it assignment} matrix of the observation $\vX_{i}$ to the cluster $k$, i.e.,  
\[
h_{ik} = \left\{
\begin{array}{cc}
1, & \mbox{if } i \in G_{k}, \\
0, & \mbox{otherwise}, \\
\end{array}
\right. \quad \mbox{for } i \in [n], \; k \in [K].
\]
Since each row of $H$ contains exactly one nonzero entry, there is one-to-one mapping (up to assignment labeling) between the partition and the assignment matrix. Thus recovery of the true clustering structure is equivalently to recovery of the associated assignment matrix.

Given the data matrix $X_{p \times n} = (\vX_{1},\dots,\vX_{n})$, the optimal estimator that maximizes the probability of recovering the clustering labels correctly is the maximum a posteriori (MAP) estimator. If the label assignment is uniformly random, then the MAP estimator is equivalent to the maximum likelihood estimator (MLE), where the log-likelihood function is given by 
\[
\ell(H, \, \mvmu_{1},\dots,\mvmu_{K}) = -{np \over 2} \log(2\pi\sigma^{2}) - {1 \over 2\sigma^{2}} \sum_{i=1}^{n} \sum_{k=1}^{K} h_{ik} \|\vX_{i} - \mvmu_{k}\|_{2}^{2}.
\]
Then the MLE corresponds to the solution of 
\begin{equation}
\label{eqn:mle_mixture}
\min_{H, \, \mvmu_{1},\dots,\mvmu_{K}} \sum_{i=1}^{n} \sum_{k=1}^{K} h_{ik} \|\vX_{i} - \mvmu_{k}\|_{2}^{2}
\end{equation}
subject to the constraint that $H$ is an assignment matrix.

Since we focus on the recovery of the true clustering structure $G_{1}^{*},\dots,G_{K}^{*}$, we may first profile the ``nuisance parameters" $\mvmu_{1},\dots,\mvmu_{K}$, whose MLEs are given by 
\[
\hat{\mvmu}_{k} = {\sum_{i=1}^{n} h_{ik} \vX_{i} \over \sum_{i=1}^{n} h_{ik}}= {1 \over |G_{k}|} \sum_{i \in G_{k}} \vX_{i},
\]
where $|G_{k}| = \sum_{i=1}^{n} h_{ik}$ denotes the cardinality of the $k$-th cluster. Substituting $\hat{\mvmu}_{k}$ into~\eqref{eqn:mle_mixture}, we see that the MLE for $H$ (and thus for $G_{1},\dots,G_{K}$) is the solution of the constrained combinatorial optimization problem:
\begin{equation}
\label{eqn:mle_assignment}
\max_{G_{1},\dots,G_{K}} \sum_{k=1}^{K} {1 \over |G_{k}|} \sum_{i,j \in G_{k}} \langle \vX_{i}, \vX_{j} \rangle \quad \mbox{subject to } \bigsqcup_{k=1}^{K} G_{k} = [n],
\end{equation}
where $\sqcup$ denotes the disjoint union.

It is now clear that, under the Gaussian mixture model, the MLE in~\eqref{eqn:mle_assignment} is equivalent to the classical $K$-means clustering method~\cite{MacQueen1967_kmeans}, which minimizes the total intra-cluster squared Euclidean distances. Since the $K$-means clustering problem is known to be worst-case $\mathsf{NP}$-hard~\cite{Dasgupta2007,MahajanNimbhorkarVaradarajan2009}, one can expect that a polynomial-time algorithm for computing the MLE of the clustering structure with exact solutions only exists in certain cases. Because of this computational barrier of the original $K$-means problem, various computationally tractable approximation algorithms are proposed in literature. 

A widely used algorithm for solving the $K$-means is Lloyd's algorithm~\cite{Lloyd1982_TIT}, which is an iterative algorithm that sequentially refines the partition structure to ensure that the $K$-means objective function is monotonically decreasing. Lloyd's algorithm has a similar nature as the classical expectation-maximization (EM) algorithm~\cite{Dempster77maximumlikelihood} in that, while the EM implicitly performs soft clustering at every E-step, Lloyd’s algorithm does hard clustering at each iteration via the Voronoi diagram. 

Given a suitable initialization (such as the spectral clustering method~\cite{KannanVempala2009}), it is shown in~\cite{LuZhou2016} that the clustering error for Lloyd's algorithm converges to zero exponentially fast, provided that 
\begin{equation}
\label{eqn:separation_lower_bound_LuZhou}
\Delta^{2} := \min_{1 \leq k \neq l \leq K} \|\mvmu_{k} - \mvmu_{l}\|_{2}^{2} \geq C \sigma^{2} {K n \over \underline{n}} \Big( 1 \vee {K p \over n} \Big),
\end{equation}
where $\underline{n} = \min_{k \in [K]} |G^{*}_{k}|$ is the minimal cluster size and $a \vee b = \max(a,b)$.

Separation lower bound in~\eqref{eqn:separation_lower_bound_LuZhou} is not sharp (in the high-dimensional setting when $p \gg n$). In the simplest symmetric two-component Gaussian mixture model: 
\[
\vX_{i} = \mvmu \eta_{i} + \vepsilon_{i},
\]
where $\eta_{i} = 1, i \in G_{1}^{*}$ and $\eta_{i} = -1, i \in G_{2}^{*}$, \cite{Ndaoud2019} proposes a simple iterative thresholding algorithm that achieves the sharp threshold on $\|\mvmu\|_{2}^{2}$ for exact recovery with high probability, which is given by 
\begin{equation}
\label{eqn:threshold_threshold_K=2}
\sigma^{2} \left( 1+ \sqrt{1 + {2p \over n \log{n}}} \right) \log{n}.
\end{equation}
It should be noted that the algorithm in~\cite{Ndaoud2019} critically depends on the symmetry of the Gaussian centers (i.e., $\mvmu$ and $-\mvmu$) and it is structurally difficult to extend such algorithm with maintained statistical optimality to a general $K$-component Gaussian mixture model without assuming the centers are equally spaced.

Another active line of research focuses on various convex relaxed versions of the $K$-means problem that is solvable in polynomial-time~\cite{PengWei2007_SIAMJOPTIM,MixonVillarRachel2016,LiLiLingStohmerWei2017,FeiChen2018,Royer2017_NIPS,GiraudVerzelen2018,BuneaGiraudRoyerVerzelen2016}. The best known rate of convergence achieved by the semidefinite programming (SDP) relaxed $K$-means for the Gaussian mixture model~\eqref{eqn:clustering_model} is given by~\cite{GiraudVerzelen2018}. Specifically, it is shown therein that misclassification errors of the SDP originally proposed in~\cite{PengWei2007_SIAMJOPTIM} for relaxing the $K$-means has the exponential rate of convergence $\exp(-C \cdot \SNR^{2})$, where the signal-to-noise ratio is defined as 
\begin{equation}
\label{eqn:exponential_rate_SDP_Kmeans_GiraudVerzelen}
\SNR^{2} = {\Delta^{2} \over \sigma^{2}} \wedge {\underline{n} \Delta^{4} \over p \sigma^{4}} \geq c {n \over \underline{n}}
\end{equation}
and $a \wedge b = \min(a,b)$. In particular, the exponential rate implies that exact recovery is achieved by the SDP relaxed $K$-means with high probability in the equal cluster size case $\underline{n} = n / K$ if minimal separation of cluster centers satisfies the lower bound
\begin{equation}
\label{eqn:separation_lower_bound_GiraudVerzelen}
\Delta^{2} \geq C \sigma^{2} \left( 1 \vee \sqrt{K p \over n \log{n}} \right) \log{n}.
\end{equation}

Now comparing~\eqref{eqn:separation_lower_bound_GiraudVerzelen} with the optimal exact results~\eqref{eqn:threshold_threshold_K=2} in the special symmetric two-component Gaussian mixture model, it is natural to ask the following question: 
\begin{quote}{\it does the SDP relaxed $K$-means clustering method achieve a sharp threshold for exact recovery of the general $K$-component Gaussian mixture model?}\end{quote}

To the best knowledge of ours, this is an open question in literature. In this paper, we provide an affirmative answer to this question: we show that there is an SDP relaxation of the $K$-means clustering method (given in~\eqref{Eqn:primal_SDP_PW} below) achieving the exact recovery with high probability if $\Delta^{2} \geq (1+\alpha) \overline{\Delta}^{2}$, where 
\begin{equation}
\label{eqn:separation_lower_bound_ChenYang}
\overline{\Delta}^{2} = 4 \sigma^{2} \left( 1 + \sqrt{1 + {K p \over n \log{n}}} \right) \log{n}.
\end{equation}
In addition, if $\Delta^{2} \leq (1-\alpha) \overline{\Delta}^{2}$, then the probability of exact recovery for any estimator vanishes to zero under the equal cluster size scenario. Thus $\overline{\Delta}^{2}$ yields the information-theoretic cutoff value on the minimal separation of cluster centers for exact recovery of the $K$-component Gaussian mixture model, and the SDP relaxation for the $K$-means is minimax-optimal in the sense that sharp phase transition of the probability of wrong recovery from zero to one occurs at the critical threshold given by the $\overline \Delta^2$.

\subsection{Related work}
There is a vast literature studying the clustering problem on the Gaussian mixture model, or more generally finite mixture models. Regarding clustering labels as missing data, parameter estimation is often carried out by the EM algorithm~\cite{Dempster77maximumlikelihood,FraleyRaftery2002_JASA}. The EM algorithm has been extensively studied in the statistics and machine learning literature~\cite{chen1995,balakrishnan2017,XuHsuMaleki2016_NIPS,KlusowskiBrinda2016,YanYinSarkar2017_NIPS,DaskalakisTzamosZampetakis2017_COLT,Dwivedi2018,WuZhou2019,Dwivedi2019}. Optimal rate of convergence for estimating the mixing distribution in finite mixture models is derived in~\cite{chen1995}. Consistency of the $K$-means estimation of the clustering centers is studied in~\cite{MacQueen1967_kmeans,pollard1981}, without concerning the computational complexity. Computationally efficient algorithms for solving the $K$-means include Lloyd's algorithm~\cite{Lloyd1982_TIT,LuZhou2016} and convex relaxations~\cite{PengWei2007_SIAMJOPTIM,MixonVillarRachel2016,LiLiLingStohmerWei2017,FeiChen2018,Royer2017_NIPS,GiraudVerzelen2018,BuneaGiraudRoyerVerzelen2016,Awasthi2015_ITCS}. Other popular clustering methods include the spectral clustering~\cite{Meila01learningsegmentation,NgJordanWeiss2001_NIPS,Vempala04aspectral,Achlioptas2005McSherry,KumarKannan2010,AwasthiSheffet2012,vanLuxburg2007_spectralclustering,vonLuxburgBelkinBousquet2008_AoS} and variants of the $K$-means~\cite{ArthurVassilvitskii2007,OstrovskyRabaniSchulmanLeonard2013Swamy,NelloreWald2015_IC,ChenYang2018,ChenYang2019}. Analysis under the mixture models has also been done under other clustering models such as the stochastic ball models~\cite{NelloreWald2015_IC,Awasthi2015_ITCS,FeiChen2018}.

Parallel to the (mixture) model-based clustering framework, there are many similar methods and algorithms proposed for community detection in network data based on the stochastic block model (SBM)~\cite{HollandLaskeyLeinhardt1983_SN,DyerFrieze1989_JoA}. Successful algorithms for community detection, partial and exact recovery under the SBM have been extensively studied in literature -- these include spectral algorithms~\cite{Krzakala20935,Massoulie2014_STOC,LeiAlessandro2015_AoS}, SDP relaxations~\cite{EmmanuelBandeiraHall2016_IEEETIT,AminiLevina2018_AoS,HajekWuXu2016a_IEEETIT,HajekWuXu2016b_IEEETIT,GuedonVershynin2016_PTRF,Bandeira2018_FoCM,LiChenXu2018_convexrelax}, among others~\cite{MosselNeemanSly2016_EJP,MosselNeemanSly2016_AoAP}.

\subsection{Notation}

Let $\vone_{n}$ be the $n \times 1$ vector of all ones. For two matrices $A$ and $B$ of the same size, let $\langle A, B \rangle = \tr(A^{T} B)$ be the usual inner product. Throughout the rest of the paper, we fix the notation $n_{k} = |G_{k}^{*}|$, $m = \min_{1 \leq k \neq l \leq K}\Big\{\frac{2n_kn_l}{n_k+n_l}\Big\}$, and $\underline n=\min_{k\in [K]} n_k$ as the minimal cluster size.

\section{Main result}
\label{sec:main_result}
In this section, we state our main result on the information-theoretic cutoff value of the exact recovery of the Gaussian mixture model in~\eqref{eqn:clustering_model}. 

\begin{thm}[Separation upper bound for exact recovery via SDP relaxation]\label{thm:SDP_bound}
If there exist constants $\delta>0$ and $\beta\in(0,1)$ such that 
\begin{align*}
\log n \geq \frac{(1-\beta)^2}{\beta^2} \frac{C_1n}{m}, \qquad \delta\leq \frac{\beta^2}{(1-\beta)^2}\frac{C_2}{K}, \qquad m \geq {4 (1+\delta)^2 \over \delta^2},
\end{align*}
and
\begin{equation*}
\Delta^2 \geq \frac{4\sigma^2(1+2\delta)}{(1-\beta)^{2}} \left(1+\sqrt{1+\frac{(1-\beta)^{2}}{(1+\delta)} \frac{p}{m\log n}+ C_3R_{n} }\,\right) \,\log n
\end{equation*}
with 
\[
R_{n} = {(1-\beta)^{2} \over (1+\delta) \log{n}} \left( {\sqrt{p\log{n}} \over \underline n} + {\log{n} \over \underline n} \right),
\]
then the SDP in~\eqref{Eqn:primal_SDP_PW} achieves exact recovery with probability at least $1-C_4\,K^2\,n^{-\delta}$, where $C_i$, $i=1,2,3,4$, are universal constants.
\end{thm}

The following corollary is a direct consequence (and a special case) of Theorem~\ref{thm:SDP_bound} when the cluster sizes are equal.

\begin{cor}\label{cor:SDP_bound}
Let $\alpha > 0$, $\Delta^{2} = \min_{1 \leq k \neq l \leq K} \|\mvmu_{k} - \mvmu_{l}\|_{2}^{2}$, and $\overline{\Delta}^{2}$ be defined in~\eqref{eqn:separation_lower_bound_ChenYang}. Suppose that the cluster sizes are equal and $K \leq C_{1} \log(n) / \log\log(n)$ for some small constant $C_{1}>0$ depending only on $\alpha$. If $\Delta^{2} \geq (1+\alpha) \overline{\Delta}^{2}$, then the SDP in~\eqref{Eqn:primal_SDP_PW} achieves exact recovery with probability at least $1-C_{2} (\log{n})^{-c_3}$, where $C_{2}, c_{3}$ are constants depending only on $\alpha$.
\end{cor}

To derive a lower bound, we focus on the equal size case where clusters $\{G_k^\ast\}_{k=1}^K$ have roughly the same sizes. More precisely, recall that our unknown parameters are the cluster indicating variables $H=\{h_{ik}:\, i\in [n],\, k\in[K]\}$, and $\{n_k:\,k\in[K]\}$ are the unknown cluster sizes. Let $\delta_n=C\sqrt{K\log(n)/n}$ for some sufficiently large constant $C>0$.
Here, we consider $n_k\in[(1-\delta_n)\, n/K,\,(1+\delta_n)\, n/K]$ for $k\in[K]$ that allows a small fluctuation on the community size in establishing the lower bound. Particularly, we define the (localized) parameter space as
\begin{align*}
&\Theta\big(n, K, \Delta\big) =\Big\{ \big(\{h_{ik}\},\,\{\vmu_k\}\big):\, h_{ik}\in\{0,1\}, \,\vmu_k\in\mathbb R^p,\, \sum_{k=1}^K h_{ik} = 1, \\
&\qquad  n_k:=\sum_{i=1}^n h_{ik} \in \Big[(1-\delta_n)\,\frac{n}{K},\, (1+\delta_n)\,\frac{n}{K}\Big],\, \|\vmu_k-\vmu_l\|\geq \Delta, \,\,\forall i\in[n]  \mbox{ and }\forall (k,l)\in[K]^2,\,k\neq l\Big\}.
\end{align*}

\begin{thm}[Separation lower bound for exact recovery: equal cluster size case]\label{thm:SDP_low_bound} Let $\alpha \in (0, 1)$. If $\Delta^2\leq (1-\alpha){\overline \Delta}^{2}$ and $K \leq \log{n}$, then we have 
\begin{align*}
\inf_{\{\hat h_{ik}\}}\sup_{(H,\vmu)\in \Theta(n,K,\Delta)} \mathbb P_{(H,\vmu)}\big(\hat h_{ik} \neq h_{ik},\, i\in[n],\,k\in[K]\big) \geq 1 - c K n^{-1},
\end{align*}
where $c > 0$ is a constant depending only on $\alpha$ and the infimum is over all possible estimators $\{\hat h_{ik}\}$ for $\{ h_{ik}\}$.
\end{thm}

Corollary~\ref{cor:SDP_bound} and Theorem~\ref{thm:SDP_low_bound} together imply that in the equal cluster size case when $n_1=n_2=\cdots=n_K=\frac{n}{K}$, the SDP relaxation~\eqref{Eqn:primal_SDP_PW} for the $K$-means is minimax-optimal in the sense that sharp phase transition of the probability of wrong recovery from zero to one occurs at the critical threshold given by the $\overline \Delta^2$ in~\eqref{eqn:separation_lower_bound_ChenYang}.

\section{Semidefinite programming relaxation: primal and dual}
\label{sec:sdp_kmeans}
In this section, we describe the SDP relaxation of the $K$-means that achieves the cutoff value of the exact recovery and outline the strategy of showing that the SDP solution uniquely recovers the true clustering structure by a dual certificate argument via the primal-dual construction. We remark that similar primal-dual analyses are done in~\cite{LiLiLingStohmerWei2017,IguchiMixonPetersonVillar2015}. 

Let $A = X^{T} X$ be the affinity matrix and $B = \diag(|G_{1}|^{-1},\dots,|G_{K}|^{-1})$. Then we can reparametrize~\eqref{eqn:mle_assignment} as 
\begin{equation}
\label{eqn:Kmeans}
\max_{H} \langle A, H B H^{T} \rangle \quad \mbox{subject to } H \in \{0,1\}^{n \times K}, \; H \vone_{K} = \vone_{n},
\end{equation}
which is a mixed integer program with a nonlinear objective function~\cite{PengWei2007_SIAMJOPTIM,Hansen1997}. If the cluster centers $\mvmu_{1},\dots,\mvmu_{K}$ are properly separated, then the affinity matrix $A$ from the data has an approximate block diagonal structure (up to a permutation of the data index).

Changing variable $Z = H B H^{T}$, we observe that the $n \times n$ symmetric matrix $Z$ satisfies the following properties: 
\begin{enumerate}[(P1)]
\item positive semidefinite (psd) constraint: $Z \succeq 0$;
\item non-negative (entrywise) constraint: $Z \geq 0$, i.e., $Z_{ij} \geq 0$ for all $i, j \in [n]$;
\item unit row-sum constraint: $Z \vone_{n} = \vone_{n}$;
\item trace constraint: $\tr(Z) = K$.
\end{enumerate}
Since $Z$ is symmetric, properties (P2) and (P3) automatically ensure that $Z$ is a stochastic matrix $Z \vone_{n} = Z^{T} \vone_{n} = \vone_{n}$. Given any clustering structure $G_{1},\dots,G_{K}$, we may consider the associated cluster membership matrix: 
\begin{equation}
\label{eqn:Kmeans_membership_matrix}
Z_{ij} = \left\{
\begin{array}{cc}
1/|G_{k}| & \text{if } i, j \in G_{k} \\
0 & \text{otherwise} \\
\end{array}
\right. .
\end{equation}
Thus to recover the true clustering structure $G_{1}^{*},\dots,G_{K}^{*}$, it suffices to compare the estimated membership matrix and the true one $Z^{*}$.

After the change-of-variables, the objective function in~\eqref{eqn:Kmeans} becomes linear in $Z$. Then we use the solution $\hat{Z}$ of the following (convex) SDP to estimate $Z^{*}$: 
\begin{equation}\label{Eqn:primal_SDP_PW}
\hat{Z} = \mbox{argmax}_{Z \in \sC_{K}}  \,\, \langle A, Z \rangle,
\end{equation} 
where 
\[
\sC_{K} = \Big\{ Z \in \bR^{n \times n} \, \Big| \, Z \succeq 0, \, Z^{T} = Z, \, \tr(Z) = K, \, Z\vone_n=\vone_n, \, Z \geq 0 \Big\}.
\]

Note that the above SDP is first proposed in~\cite{PengWei2007_SIAMJOPTIM} and later studied in~\cite{GiraudVerzelen2018,ChenYang2018,ChenYang2019}. For spherical Gaussians (i.e., the noise covariance matrix is proportional to the identity matrix), since the SDP relaxation~\eqref{Eqn:primal_SDP_PW} does not require the knowledge of the noise variance $\sigma^{2}$ and the partition information other than the number of clusters $K$,  it in fact can handle the more general case of unequal cluster sizes.

\begin{rem}[Adaptation to the number of clusters $K$]
The SDP in~\eqref{Eqn:primal_SDP_PW} can be made adaptive to the unknown number of cluster $K$. When the number of clusters $K$ is unknown, the constraint $\mbox{tr}(Z)=K$ in the SDP~\eqref{Eqn:primal_SDP_PW} can be lifted to a penalization term in its objective function, i.e., we solve
\begin{equation}
\label{eqn:regularized_Kmeans}
\tilde{Z}_{\lambda} := \arg\max \{ \langle A, Z \rangle -\lambda \mbox{tr}(Z): Z \succeq 0, Z^{T} = Z, Z \vone_{n} = \vone_{n}, Z \geq 0 \},
\end{equation}
where $\lambda \geq 0$ is a regularization parameter. This is the \emph{regularized $K$-means} proposed by~
\cite{BuneaGiraudRoyerVerzelen2016,Royer2017_NIPS} and analyzed by~\cite{ChenYang2019} in the manifold clustering setting. Using the same existing argument for proving the separation upper bound in Section~\ref{sec:proof_main_result}, we see that with the $\lambda$ choice being
\begin{align}
\sigma^2(\sqrt{n}+\sqrt{p} +\sqrt{2\log n}\,)^2 +&C\beta^{-1}\sigma^2\, (n + K\log n +(1-\beta)K\delta\sqrt{p m \log n})   \notag\\
&\qquad  \leq  \lambda  \leq p\sigma^2 + \frac{\beta}{4}\,m \Delta^2,  \label{Eqn:lambda_choice}
\end{align}
then under the same conditions in Theorem 2.1, $\tilde{Z}_{\lambda} = Z^{*}$ achieves exact recovery with probability at least $1 - C K^{2} n^{-\delta}$. Note that a larger signal-to-noise ratio $\Delta^2/\sigma^2$ permits a wider allowable range for $\lambda$ to achieve exact recovery, and our conditions in Theorem~\ref{thm:SDP_bound} ensures the existence of at least one such $\lambda$.
 The idea for $\tilde{Z}_{\lambda}$ to achieve the sharp threshold (i.e., the separation upper bound) is that the SDP giving $\hat{Z}$ in~\eqref{Eqn:primal_SDP_PW} and its regularized version in~\eqref{eqn:regularized_Kmeans} have the same Lagrangian form and the dual problem. Thus we need only to extract the regime of the regularization parameter $\lambda$ in~\eqref{Eqn:lambda_choice} that ensures a successful dual certificate construction as characterized in $\hat{Z}$ (Section~\ref{sec:proof_main_result}). In particular, the dual certificate constructed for $\hat{Z}$ is a convenient choice of $\lambda^{\sharp}$ that falls into the region~\eqref{Eqn:lambda_choice} with high probability. In addition, \cite{ChenYang2019} provides a practical method for adaptively tuning this regularization parameter $\lambda$.  \qed
\end{rem}

Note that $Z^{*}$ is a rank-$K$ block diagonal matrix, and for any $Z \in \sC_{K}$, due to the psd constraint, $\tr(Z)$ equals to the nuclear norm $\|Z\|_{*}$. Then the SDP in~\eqref{Eqn:primal_SDP_PW} can be effectively viewed as a low-rank matrix denoising procedure for the data affinity matrix $A$ by finding its optimal matching from all feasible ``rank-$K$" stochastic matrices proxied by the trace constraint. 

On the other hand, the SDP solutions are not integral in general. If this is the scenario, then the standard relaxing-and-rounding paradigm~\cite{WilliamsonShmoys2011_Design_Approx_Algs} can be used to round the SDP solution back to a point in the feasible set of the original discrete optimization problem~\eqref{eqn:mle_assignment}. In our case, we can apply the $K$-means clustering to the top $K$-eigenvectors of $\hat{Z}$ as a rounding procedure to extract the estimated partition structure $\hat{G}_{1},\dots,\hat{G}_{K}$.

However, it is observed that the rounding step is not always necessary and solution to the clustering problem~\eqref{eqn:mle_assignment} can be directly recovered from solving the relaxed SDP problems when the separation of cluster centers is large, which is sometimes referred to the {\it exact recovery} or {\it hidden integrality} phenomenon~\cite{Awasthi2015_ITCS,FeiChen2018}. This motivates the question we asked earlier in Section~\ref{sec:introduction} that when and to what extend the SDP relaxation can in fact produce the exact recovery. The rest of the paper is devoted to characterize the precise cutoff value on the separation of cluster centers that yields the exact recovery.

\subsection{Dual problem}
To analyze the exact recovery property of $\hat{Z}$, we first derive the dual problem for the (primal) SDP problem in~\eqref{Eqn:primal_SDP_PW}. Let 
\begin{align*}
\cL(Z, Q, \lambda, \valpha, B) & = \tr(AZ) + \tr(QZ) + \lambda(K-\tr(Z)) + \valpha^{T} \left( \vone_{n} - {Z+Z^{T} \over 2} \vone_{n} \right) + \tr(BZ) \\
& = (\lambda K + \valpha^{T} \vone_{n}) + \tr\left\{\left[A+Q-\lambda\Id_{n}+B-{1\over2}(\vone_{n}\valpha^{T}+\alpha\vone_{n}^{T})\right] Z \right\}
\end{align*}
be the Lagrangian function, where $Q_{n \times n} \succeq 0$, $\valpha_{n \times 1}=(\alpha_1,\ldots,\alpha_n)^T$, $B_{n \times n} \geq 0$, and $\lambda \in \bR$ are the Lagrangian multipliers. Consider the max-min problem: 
\[
\max_{Z \in \bR^{n \times n}} \,\, \min_{Q \succeq 0, \lambda \in \bR, \valpha \in \bR^{n}, B \geq 0} \cL(Z, Q, \lambda, \alpha, B),
\]
where the maximum over $Z$ is unconstrained. If $Z$ is not primal feasible for the SDP problem~\eqref{Eqn:primal_SDP_PW}, then 
\[
\min_{Q \succeq 0, \lambda \in \bR, \valpha \in \bR^{n}, B \geq 0} \,\, \cL(Z, Q, \lambda, \alpha, B) = -\infty.
\]
For example, consider $\tr(Z) \neq K$ and choose $\lambda = -{c \over K-\tr(Z)}$ with an arbitrarily large $c > 0$. On the other hand, if $Z$ is feasible for~\eqref{Eqn:primal_SDP_PW}, then 
\[
\tr(AZ) \leq \min_{Q \succeq 0, \lambda \in \bR, \valpha \in \bR^{n}, B \geq 0} \,\, \cL(Z, Q, \lambda, \alpha, B),
\]
where the equality is attained if for example $Q = B = 0$. Then,  
\begin{align*}
\max_{Z \in \sC_{K}} \tr(AZ) & \leq \max_{Z \in \bR^{n \times n}} \,\, \min_{Q \succeq 0, \lambda \in \bR, \valpha \in \bR^{n}, B \geq 0} \,\, \cL(Z, Q, \lambda, \alpha, B) \\
& \leq \min_{Q \succeq 0, \lambda \in \bR, \valpha \in \bR^{n}, B \geq 0}  \,\,\max_{Z \in \bR^{n \times n}}  \,\, \cL(Z, Q, \lambda, \alpha, B).
\end{align*}
Similarly, if $A+Q-\lambda\Id_{n}+B-{1\over2}(\vone_{n}\valpha^{T}+\alpha\vone_{n}^{T}) \neq 0$, then 
\[
\max_{Z \in \bR^{n \times n}}  \,\, \tr\left\{\left[A+Q-\lambda\Id_{n}+B-{1\over2}(\vone_{n}\valpha^{T}+\alpha\vone_{n}^{T})\right] Z \right\} = \infty,
\]
which is avoided by the minimization over the Lagrangian multipliers. Thus with $Q = \lambda\Id_{n}+{1\over2}(\vone_{n}\valpha^{T}+\alpha\vone_{n}^{T})-B-A$, we have 
\begin{align*}
\max_{Z \in \sC_{K}} \,\, \tr(AZ) & \leq \min_{\lambda \in \bR,\valpha \in \bR^{n},B \in \bR^{n \times n}} \,\, \left\{ \lambda K + \valpha^{T} \vone_{n} : B\geq0,\lambda\Id_{n}+{1\over2}(\vone_{n}\valpha^{T}+\alpha\vone_{n}^{T})-B-A\succeq0 \right\},
\end{align*}
which is the weak duality between the primal SDP problem~\eqref{Eqn:primal_SDP_PW} and its dual problem: 
\begin{equation}\label{Eqn:dual_SDP_PW}
\begin{aligned}
&\mbox{min}_{\lambda\in\bR,\valpha\in\bR^{n},B\in\bR^{n\times n}} \,\, \{ \lambda K + \valpha^{T} \vone_{n} \}  \\
\qquad&\mbox{subject to }\ B \geq 0, \\
& \quad  \qquad\qquad \lambda\Id_{n}+{1\over2}(\vone_{n}\valpha^{T}+\alpha\vone_{n}^{T})-B-A\succeq0.
\end{aligned}
\end{equation}
Moreover, the duality gap is given by 
\begin{equation}\label{Eqn:dual_SDP_duality_gap}
\begin{aligned}
\lambda K + \valpha^{T} \vone_{n} - \tr(AZ) &= \lambda \tr(Z) + \alpha^{T} {Z+Z^{T} \over 2} \vone_{n} - \tr(AZ) \\
&= \tr \left\{ \left[ \lambda \Id_{n} + {1\over2}(\vone_{n}\valpha^{T}+\alpha\vone_{n}^{T}) - A - B \right] Z \right\} + \tr(BZ) \\
& \geq \tr(BZ) \geq 0.
\end{aligned}
\end{equation}

\subsection{Optimality conditions: primal-dual construction}
\label{subsec:primal-dual_construction}
Let $\vone_{G_k^\ast}$ be the $n \times 1$ vector such that it is equal to $\vone_{n_{k}}$ on $G_k^\ast$ and zero otherwise. To show that 
\begin{align}\label{Eqn:True_Z}
Z^\ast =\begin{bmatrix}
\frac{1}{n_1}J_{n_1} & 0 & \cdots & 0\\
0 & \frac{1}{n_2} J_{n_2} & \cdots & 0\\
\vdots & \ddots & \cdots & \vdots\\
0 & \cdots & 0 & \frac{1}{n_K}J_{n_K}
\end{bmatrix}=\sum_{k=1}^K \frac{1}{n_k} \vone_{G_k^\ast}\vone_{G_k^\ast}^T
\end{align}
is the solution of the primal SDP problem~\eqref{Eqn:primal_SDP_PW}, we need the duality gap~\eqref{Eqn:dual_SDP_duality_gap} is zero at $Z = Z^\ast$. To this end, we need to construct a {\it dual certificate} $(\lambda, \valpha, B)$ such that: 
\begin{enumerate}[(C1)]
\item $B \geq 0$;
\item $W_{n} := \lambda \Id_{n} + {1\over2}(\vone_{n}\valpha^{T}+\alpha\vone_{n}^{T}) - A - B \succeq 0$;
\item $\tr(W_{n} Z^{\ast}) = 0$;
\item $\tr(B Z^{\ast}) = 0$.
\end{enumerate}
Note that (C1) and (C2) are dual feasibility constraints, while (C3) and (C4) are the optimality conditions (i.e., complementary slackness) corresponding to the zero duality gap in~\eqref{Eqn:dual_SDP_duality_gap}. In particular, (C4) implies that $B_{G_{k}^{\ast}G_{k}^{\ast}} = 0$ for all $k \in [K]$.

To ensure that $Z^{\ast}$ is the {\it unique} solution of the SDP problem~\eqref{Eqn:primal_SDP_PW}, we observe that $Z^{\ast}$ is the only feasible matrix to the SDP~\eqref{Eqn:primal_SDP_PW} satisfying the block diagonal structure
\begin{align*}
\begin{bmatrix}
Z^{(1)} & 0 & \cdots & 0\\
0 & Z^{(2)} & \cdots & 0\\
\vdots &  \ddots & \cdots & \vdots\\
0 & \cdots & 0 & Z^{(K)}
\end{bmatrix},
\end{align*}
i.e., $Z_{G_k^\ast G_l^\ast}=0$ for all distinct pair $(k,l)\in[K]^2$. Indeed, since each block $Z^{(k)}$ satisfies $Z^{(k)}\vone_{n_k} =\vone_{n_k}$ and is psd, $(1, {n_k}^{-1/2} \vone_{n_k})$ is one eigenvalue-eigenvector pair of $Z^{(k)}$ and the trace of $Z^{(k)}$ is at least $1$. On the other hand, due to the trace constraint $\sum_{k=1}^K \tr(Z^{(k)}) = \tr(Z)=k$, we then must have $\tr(Z^{(k)})=1$. In addition, $1$ is its only nonzero eigenvalue with eigenvector ${n_k}^{-1/2}\vone_{n_k}$. Consequently, $Z^{(k)}$ must take the form of $n_k^{-1} J_{n_k}$. 

Given the above block diagonal structure and $\tr(B Z^{\ast}) = 0$, we conclude that $Z^{\ast}$ is the unique solution to the SDP~\eqref{Eqn:primal_SDP_PW} if 
\begin{enumerate}[(C1)]
\setcounter{enumi}{4}
\item $B_{G_{k}^{\ast}G_{l}^{\ast}} > 0$ for all distinct pair $(k,l)\in[K]^2$,
\end{enumerate}
in addition to the optimality conditions (C1)-(C4).

\section{Proof of Theorem~\ref{thm:SDP_bound}}
\label{sec:proof_main_result}
In this section, we show that a dual certificate described in Section~\ref{subsec:primal-dual_construction} can be successfully constructed with high probability, thus proving Theorem~\ref{thm:SDP_bound}. First, observe that $W_{n} \succeq 0$ and $\tr(W_{n} Z^{\ast}) = 0$ imply that 
\begin{equation}
\label{eqn:eigensystem_Wn}
W_{n} \vone_{G_{k}^{*}} = 0 \quad \text{for all } k \in [K].
\end{equation}
The last display together with $B_{G_{k}^{\ast}G_{k}^{\ast}} = 0$ imply that for each distinct pair $(k,l)\in[K]^{2}$, 
\begin{align}
\label{eqn:construct_alpah_eqn1}
\lambda \vone_{n_{k}} + {1 \over 2} \vone_{n_{k}} ( \sum_{i \in G_{k}^{\ast}} \alpha_{i} ) + {1 \over 2} \valpha_{G_{k}^{\ast}} n_{k} & = A_{G_{k}^{\ast}G_{k}^{\ast}} \vone_{n_{k}}, \\
\label{eqn:construct_alpah_eqn2}
{1 \over 2} \vone_{n_{l}} ( \sum_{i \in G_{k}^{\ast}} \alpha_{i} ) + {1 \over 2} \valpha_{G_{l}^{\ast}} n_{k} - A_{G_{l}^{\ast}G_{k}^{\ast}} \vone_{n_{k}} & = B_{G_{l}^{\ast}G_{k}^{\ast}} \vone_{n_{k}},
\end{align}
where $\valpha^{T} = (\valpha_{G_{1}^{\ast}}^{T},\dots,\valpha_{G_{K}^{\ast}}^{T})$. From~\eqref{eqn:construct_alpah_eqn1}, we get 
\[
\sum_{i \in G_{k}^{\ast}} \alpha_{i} = {1 \over n_{k}} \vone_{n_{k}}^{T} A_{G_{k}^{\ast}G_{k}^{\ast}} \vone_{n_{k}} - \lambda.
\]
Substituting the last equation back into~\eqref{eqn:construct_alpah_eqn1}, we get
\begin{align}\label{Eqn:alpha_def}
\valpha_{G_{k}^{\ast}} ={2 \over n_{k}} A_{G_{k}^{\ast}G_{k}^{\ast}} \vone_{n_{k}} - {\lambda \over n_{k}} \vone_{n_{k}} - {1 \over n_{k}^{2}} \vone_{n_{k}} (\vone_{n_{k}}^{T} A_{G_{k}^{\ast}G_{k}^{\ast}} \vone_{n_{k}}).
\end{align}
Next we construct a solution of $B$ for~\eqref{eqn:construct_alpah_eqn2}. For $k \neq l$, we have 
\[
B_{G_{l}^{\ast}G_{k}^{\ast}} \vone_{n_{k}} = -{n_{l}+n_{k} \over 2 n_{l}} \lambda \vone_{n_{l}} + {1 \over 2n_{k}} (\vone_{n_{k}}^{T} A_{G_{k}^{\ast}G_{k}^{\ast}} \vone_{n_{k}}) \vone_{n_{l}} + {n_{k} \over n_{l}} A_{G_{l}^{\ast}G_{l}^{\ast}} \vone_{n_{l}} - {n_{k} \over 2 n_{l}^{2}} (\vone_{n_{l}}^{T} A_{G_{l}^{\ast}G_{l}^{\ast}} \vone_{n_{l}}) \vone_{n_{l}} - A_{G_{l}^{\ast}G_{k}^{\ast}} \vone_{n_{k}}.
\]
In particular, for $j \in G_{l}^{\ast}$, 
\begin{align}
\nonumber
[B_{G_{l}^{\ast}G_{k}^{\ast}} \vone_{n_{k}}]_{j} & = -{n_{l}+n_{k} \over 2 n_{l}} \lambda + {1 \over 2n_{k}} \sum_{s,t \in G_{k}^{\ast}} \vX_{s}^{T} \vX_{t} + {n_{k} \over n_{l}} \sum_{t \in G_{l}^{\ast}} \vX_{j}^{T} \vX_{t} - {n_{k} \over 2 n_{l}^{2}} \sum_{s,t \in G_{l}^{\ast}} \vX_{s}^{T} \vX_{t} - \sum_{t \in G_{k}^{\ast}} \vX_{j}^{T} \vX_{t} \\
\nonumber
& = -{n_{l}+n_{k} \over 2 n_{l}} \lambda + {n_{k} \over 2} (\overline{\vX}_{k}^{T} \overline{\vX}_{k} - \overline{\vX}_{l}^{T} \overline{\vX}_{l}) + n_{k} \vX_{j}^{T} (\overline{\vX}_{l} - \overline{\vX}_{k}) \\
\label{eqn:construct_B_row_sum}
& = -{n_{l}+n_{k} \over 2 n_{l}} \lambda + {n_{k} \over 2} (\|\overline{\vX}_{k} - \vX_{j}\|_{2}^{2} - \|\overline{\vX}_{l} - \vX_{j}\|_{2}^{2}),
\end{align}
where $\overline{\vX}_{k} = n_{k}^{-1} \sum_{i \in G_{k}^{\ast}} \vX_{i}$ is the empirical mean of data points in the $k$-th cluster. Without loss of generality, we may take a symmetric $B$ (i.e., $B^{T} = B$) and then construct $B$ as block-wise rank-one matrix satisfying the above row sum constraint~\eqref{eqn:construct_B_row_sum}: 
\begin{align}\label{eqn:B_form}
B^{\sharp}_{G_l^\ast G_k^\ast} = \frac{B_{G_l^\ast G_k^\ast} \vone_{G_k^\ast} \vone_{G_l^\ast}^T B_{G_l^\ast G_k^\ast}}{ \vone_{G_l^\ast}^T B_{G_l^\ast G_k^\ast} \vone_{G_k^\ast}}\,\, \mbox{for each distinct pair $(k,l)\in[K]^2$, and } B^{\sharp}_{G_k^\ast G_k^\ast}=0. 
\end{align}
For notational simplicity, let us denote the column sums and row sums of matrix $B_{G^\ast_kG^\ast_l}$ in~\eqref{eqn:B_form} by $\vc^{(k,l)}=\big(c^{(k,l)}_j:\, j\in G^\ast_l\big)$ and $\vr^{(k,l)}=\big(r^{(k,l)}_i:\, i\in G^\ast_k\big)$, respectively. In addition, by letting $t^{(k,l)}=\sum_{j\in G^\ast_l}c^{(k,l)}_j =\sum_{i\in G^\ast_k} r^{(k,l)}_i$ be the total sum, then the construction in~\eqref{eqn:B_form} becomes $[B_{G_l^\ast G_k^\ast}]_{ij}= r^{(k,l)}_i c^{(k,l)}_j/t^{(k,l)}$. For convenience, we also define $r^{(k,k)}_i=c^{(k,k)}_j=t^{(k,k)}=0$ for all $i,j\in G^\ast_{k}$, so that $B_{G_k^\ast G_k^\ast} = 0$ for all $k\in[K]$ (define $0/0=0$).

Recall that to ensure uniqueness, we need to choose $\lambda$ such that $B^{\sharp}_{G_{k}^{\ast}G_{l}^{\ast}} > 0$ for all distinct pair $(k,l)\in[K]^{2}$, which is, in view of~\eqref{eqn:construct_B_row_sum}, guaranteed whenever 
\begin{equation}
\label{eqn:uniqueness_guarantee}
\lambda < \min_{1 \leq k \neq l \leq K} \left\{ {n_{l} n_{k} \over n_{l}+n_{k}} \min_{j \in G_{l}^{\ast}} (\|\overline{\vX}_{k} - \vX_{j}\|_{2}^{2} - \|\overline{\vX}_{l} - \vX_{j}\|_{2}^{2}). \right\}.
\end{equation}
On the other hand, we require that $\lambda$ is not too small since $W_{n} = \lambda \Id_{n} + {1\over2}(\vone_{n}\valpha^{T}+\valpha\vone_{n}^{T}) - A - B \succeq 0$. 
To identify the right $\lambda$, we will employ the following lemma that provides some high probability lower bounds that will be useful for bounding from below the column sums $\{c^{(k,l)}_j:\,j\in G^\ast_l\}$ and row sums $\{r^{(k,l)}_i:\,i\in G^\ast_k\}$ under proper separation conditions on the Gaussian centers. Recall that  
\[
\Delta=\min_{1 \leq k \neq l \leq K} \|\vmu_k-\vmu_l\| \quad \mbox{and} \quad m = \min_{1 \leq k \neq l \leq K}\Big\{\frac{2n_kn_l}{n_k+n_l}\Big\}.
\]
Note that $\Delta$ is the minimum separation between the cluster centers and $m$ quantifies the ``minimum" cluster size in the pairwise sense.

\begin{lem}[Separation bound on the Gaussian centers]\label{lem:separation_bound}
Let $\delta>0$ and $1>\beta >0$. If there exists a sufficiently large universal constant $c_{1} > 0$ such that 
\begin{equation}
\label{eqn:separation_bound}
\Delta^2 \geq \frac{4\sigma^2(1+2\delta)}{(1-\beta)^{2}} \left(1+\sqrt{1+\frac{(1-\beta)^{2}}{(1+\delta)} \frac{p}{m\log n}+ c_1R_{n} }\,\right) \,\log n
\end{equation}
with 
\[
R_{n} = {(1-\beta)^{2} \over (1+\delta) \log{n}} \left( {\sqrt{p\log(nK)} \over \underline n} + {\log(nK) \over \underline n} \right),
\]
then as long as $m \geq 4 (1+\delta^{-1})^2$,
\begin{align*}
\mathbb P\Big(&\|\vX_i - \overline{\vX}_l\|^2 -  \|\vX_i - \overline{\vX}_k\|^2 \geq \frac{n_k+n_l}{n_k n_l}\sigma^{2}p +\beta \,\|\vmu_k-\vmu_l\|^2 - r_{kl},\\
&\qquad\qquad\qquad\qquad \mbox{ for all distinct pairs $(k,l)\in[K]^2$ and $i\in G^\ast_k$}\Big) \leq {K^{2} \over n^{\delta}} + {8 \over n},\\
\mbox{where}&\ \ r_{kl} = 2\sigma \sqrt{2\log(nK) \over n_{l}} \|\vmu_k-\vmu_l\| + 2 \sigma^2\frac{n_k+n_l}{n_kn_l} \sqrt{2p\log(nK)} + {4 \sigma^2\over n_{k}} \log(nK).
\end{align*}
\end{lem}
If the conditions of Lemma~\ref{lem:separation_bound} holds, then according to this lemma we may choose 
\begin{equation}\label{Eqn:lambda_def}
\lambda^\sharp = p\sigma^2 + \frac{\beta}{4} m\Delta^2, 
\end{equation}
so that  it holds with probability at least $1-2n^{-\delta}-10n^{-1}$ that for all $k,l\in[K]$, $k\neq l$, $i\in G^\ast_k$, $j\in G^\ast_l$, 
\begin{align}\label{eqn:t_lower_bound}
r^{(k,l)}_i \geq  \frac{\beta}{2} n_l\, \|\vmu_k-\vmu_l\|^2,\ \  \ c^{(k,l)}_j\geq   \frac{\beta}{2} n_k\, \|\vmu_k-\vmu_l\|^2 \ \ \ \mbox{and} \ \ \
t^{(k,l)} \geq  \frac{\beta}{2} n_kn_l\,\|\vmu_k-\vmu_l\|^2,
\end{align}
as long as $m \geq 4 (1+\delta^{-1})^2$. This implies $B^{\sharp}_{G^\ast_kG^\ast_l}>0$ for any distinct pair $(k,l)\in[K]^2$. We fix such a choice for $\lambda$ in the rest of the proof.

Denote $\Gamma_{K} = \mbox{span}\{\vone_{G_{k}^{\ast}} : k \in [K]\}^{\bot}$ be the orthogonal complement of the linear subspace of $\bR^{n}$ spanned by the vectors $\vone_{G_{1}^{\ast}},\dots,\vone_{G_{K}^{\ast}}$. In view of~\eqref{eqn:eigensystem_Wn}, we see that $\{\vone_{G_{k}^{\ast}} : k \in [K]\}$ are eigenvectors of $W_{n}$ associated to the zero eigenvalues. Thus to ensure $W_{n} \succeq 0$, we only need to check that: for any $\vv = (v_{1},\dots,v_{n})^{T} \in \Gamma_{K}$ such that $\|\vv\|_{2} = 1$, 
\[
\vv^{T} W_{n} \vv \geq 0.
\]
Our next task is to derive a high probability lower bound for the quadratic form $\vv^{T} W_{n} \vv$. Plugging the definition of $W_{n}$, we write 
\[
\vv^{T} W_{n} \vv = \lambda\|\vv\|^2 + {1\over2} (\vv^{T}\vone_{n}\valpha^{T}\vv+\vv^{T}\valpha\vone^{T}\vv) - \sum_{k,l=1}^{K} \sum_{i \in G_{k}^{\ast}} \sum_{j \in G_{l}^{\ast}} \vX_{i}^{T} \vX_{j} v_{i} v_{j} - \vv^{T}B^{\sharp}\vv.
\]
Since $\vv^{T} \vone_{G_{k}^{\ast}} = 0$ or $\sum_{i \in G_{k}^{\ast}} v_{i} = 0$ for all $k \in [K]$ and $\vv \in \Gamma_{K}$, we get 
\[
\vv^{T} W_{n} \vv = \lambda \|\vv\|^2 - S(\vv) - T(\vv),
\]
where $S(\vv) := \| \sum_{k=1}^{K} \sum_{i \in G_{k}^{\ast}} \vX_{i} v_{i} \|_{2}^{2}$ and $T(\vv) = \vv^{T}B^{\sharp}\vv$. Recall the clustering model~\eqref{eqn:clustering_model}: $\vX_{i} = \vmu_{k}+\vepsilon_{i}$ for $i \in G_{k}^{\ast}$, we have 
\[
\sum_{i \in G_{k}^{\ast}} \vX_{i} v_{i} = \vmu_{k} \sum_{i \in G_{k}^{\ast}} v_{i} + \sum_{i \in G_{k}^{\ast}} \vepsilon_{i} v_{i} = \sum_{i \in G_{k}^{\ast}} \vepsilon_{i} v_{i}.
\]
so that 
\begin{align*}
S(\vv) = \sum_{i=1}^n\sum_{j=1}^n  \vepsilon_i \vepsilon_j^{T} v_i v_j
\end{align*}
is a quadratic form in $\vv$. Therefore, for each $\vv\in \Gamma_K$ satisfying $\|\vv\|=1$, $S(\vv)$ can be bounded by the largest singular value of the Gram matrix $G_n = \big\{\vepsilon_i \vepsilon_j^{T}: i,j\in[n]\big\}$, so that $S(\vv)=\vv^T \mathcal E^T \mathcal E \vv\leq \|\mathcal E^T \mathcal E\|_{\op}=\|\mathcal E\|_{\op}^2$, where matrix 
$$\mathcal E = (\vepsilon_1,\ \vepsilon_2,\ \ldots, \ \vepsilon_n)\in \mathbb R^{p\times n}$$
has i.i.d. $N(0,\sigma^2)$ entries. Applying Lemma~\ref{lem:bound_Sv}, we can reach
\begin{align*}
\mathbb P\Big( \max_{\vv \in \Gamma_{K},\,\|\vv\|=1} S(\vv) \geq \sigma^2(\sqrt{n}+\sqrt{p} +\sqrt{2t}\,)^2 \Big) \leq  e^{-t},\quad\forall t>0.
\end{align*}

Now we analyze the last term $T(\vv)$. 
\begin{lem}[Bound on $T(\vv)$]
\label{lem:bound_Tv}
Assume the separation condition~\eqref{eqn:separation_bound} in Lemma~\ref{lem:separation_bound} and consider the choice of $\lambda$ as~\eqref{Eqn:lambda_def}.
We have for any $\delta > 0$, 
\begin{align*}
\mathbb P\bigg(  |T(\vv)|  \geq  C\beta^{-1} \sigma^2 \big(n + K\log n + (1-\beta) K \delta \sqrt{m p \log n}\big)\,\|\vv\|^2,& \; \forall \vv\in \Gamma_K 
 \ \bigg|\ \{\bar\vepsilon_k:\,k\in[K]\}\bigg)\\
 &\leq 4K^2n^{-\delta} + 10n^{-1}.
\end{align*}
\end{lem}

\noindent By combining previous bounds on $|S(\vv)|$ and $|T(\vv)|$ together, we obtain
\begin{align*}
\mathbb P \bigg(\langle \vv, \, W_n\vv\rangle \leq \lambda - \sigma^2(\sqrt{n}+\sqrt{p} +\sqrt{2\log{n}}\,)^2 - & C\beta^{-1} \sigma^2 \big(n + K\log n + (1-\beta) K \delta \sqrt{m p \log n}\big), \\
 & \forall \vv\in \Gamma_K, \,\|\vv\|=1\bigg) \leq (5K^2+1)\,n^{-\delta}. 
\end{align*}
Combining this with our constructions~\eqref{Eqn:lambda_def} for $\lambda^\sharp$,~\eqref{Eqn:alpha_def} for $\valpha^\sharp$ and~\eqref{eqn:B_form} for $B^\sharp$ and all previous analysis, we obtain that $(\lambda^\sharp,\valpha^\sharp,B^\sharp)$ will be a dual certificate that satisfies (1)--(5) 
with probability at least $1-(5K^2+1)\,n^{-\delta}$ if
\begin{align}
&\sigma^2(\sqrt{n}+\sqrt{p} +\sqrt{2\log n}\,)^2 +C\beta^{-1}\sigma^2\, (n + K\log n +(1-\beta)K\delta\sqrt{p m \log n})   \notag\\
&\qquad   \leq p\sigma^2 + \frac{\beta}{4}\,m \frac{4\sigma^2(1+2\delta)}{(1-\beta)^{2}} \left(1+\sqrt{1+\frac{(1-\beta)^{2}}{(1+\delta)} \frac{p}{m\log n}+ c_1R_{n} }\,\right) \,\log n.\label{Eqn:To_prove}
\end{align}
which is true if for some universal constants $C,c>0$, 
\begin{align*}
\log n \geq \frac{(1-\beta)^2}{\beta^2} \frac{Cn}{m},\quad\mbox{and}\quad \delta\leq \frac{\beta^2}{(1-\beta)^2}\frac{c}{K}.
\end{align*}

\section{Proof of Theorem~\ref{thm:SDP_low_bound}}
The first step is to reduce the worst-case misclassification risk to the average-case risk by putting a prior $\pi^H$ over $H=\{h_{ik}\}$ with $(h_{i1},\ldots,\,h_{iK})$ being i.i.d.~ following the multinomial distribution with one trial and probability vector $(n/K,\ldots,n/K)$. By the classical Chernoff bound we have 
\begin{align}\label{eqn:chernoff}
\mathbb P_{\pi^H}\Big( n_k:=\sum_{i=1}^n h_{ik} \in \Big[(1-\delta_n)\,\frac{n}{K},\, (1+\delta_n)\,\frac{n}{K}\Big],\, k\in[K]\Big) \geq 1-n^{-1},
\end{align}
by choosing the constant $C$ in $\delta_n$ large enough. 
As a consequence, we have
\begin{align*}
&\inf_{\{\hat h_{ik}\}}\sup_{(H,\vmu)\in\Theta(n,K,\Delta)} \mathbb P_{(H,\vmu)}\big(\hat h_{ik} \neq h_{ik},\, i\in[n],\,k\in[K]\big) \\
\geq& \, \inf_{\{\hat h_{ik}\}}\sup_{\|\vmu_k-\vmu_l\|\geq \Delta, \forall (k,l)\in[K]^2, k\neq l} \bE_{\pi^H} \mathbb P_{(H,\vmu)}\big(\hat h_{ik} \neq h_{ik},\, i\in[n],\,k\in[K]\big) - n^{-1}.
\end{align*}
Conditioning on the event that $n_1+n_2$ points belong to the first two clusters, the problem of correctly classifying all $n$ samples into $K$ clusters is always not easier than correctly classifying the $n_1+n_2$ points into the first and second clusters, that is,
\begin{align*}
P_{(H,\vmu)}\big(\hat h_{ik} \neq h_{ik},\, i\in[n],\,k\in[K]\big) \geq P_{(H,\vmu)}\big(\hat h_{ik} \neq h_{ik},\, i\in G_1\cup G_2,\,k\in[2]\big),
\end{align*}
where recall that $G_k=\{i\in[n]:\, h_{ik}=1\}$ denote the $k$-th cluster.
Now we apply the following minimax lower bound Lemma~\ref{lem:lower_bound__K=2} proved in Section~\ref{Sec:minimax_K_2} below for two clusters $G_1$ and $G_2$ conditioning on their total sizes $n_1+n_2$,
\begin{align*}
\inf_{\{\hat h_{ik},i\in G_1\cup G_2, k=1,2\}}  \sup_{\|\vmu_1-\vmu_2\|\geq \Delta} \bE_{\tilde\pi^{12}} P_{(H,\vmu)}\big(\hat h_{ik} \neq h_{ik},\, i\in G_1\cup G_2,\,k\in[2]\big) \geq 1-cK/n,
\end{align*}
for some $c>0$, where $\tilde \pi^{12}$ denote the conditional prior distribution of $\{\hat h_{ik},i\in G_1\cup G_2, k=1,2\}$ given the total sample size $n_1+n_2$ of $G_1\cup G_2$, which is uniform over $\{1,2\}^{n_1+n_2}$. Here we have used the high probability bound~\eqref{eqn:chernoff} so that with probability at least $1-n^{-1}$, the separation $\Delta$ satisfies 
\[
\Delta^2 \leq 4(1-\alpha/2)\sigma^2\left(1+\sqrt{1+\frac{2p}{(n_1+n_2)\log (n_1+n_2)}}\right) \log n.
\]
Note that the proof of Lemma~\ref{lem:lower_bound__K=2} also reduces the worst-case bound to the average-case bound, where the prior on the cluster label is uniform as the conditional distribution $\tilde\pi^{12}$ given the total size $n_1+n_2$. 
Putting all pieces together and using $K \leq \log{n}$ give a proof of the claimed result.

\subsection{Lower bound for $K=2$}\label{Sec:minimax_K_2}
Now we prove an information-theoretic limit for exact recovery of clusters labels in a symmetric two-component Gaussian mixture model, 
\begin{align}
\label{eqn:2component_sym_mixtureGaussians}
\vX_i = \eta_i \vmu  + \sigma\vepsilon_i, \quad \vepsilon_i \overset{\mbox{i.i.d.}}{\sim} N(0, I_p),\quad i=1,\ldots,n,
\end{align}
where $\vmu$ and $-\vmu$ are unknown centers of the two symmetric Gaussian components, and $\eta_i\in\{-1,1\}$ is the label for the $i$th observation indicating which component it comes from.  

\begin{lem}[Separation lower bound for exact recovery: $K=2$]
\label{lem:lower_bound__K=2}
Let $\alpha \in (0, 1)$. Consider the symmetric two-component Gaussian mixture model in~\eqref{eqn:2component_sym_mixtureGaussians} with an independent Rademacher prior distribution on $\eta_{i}$. If $\Delta^{2} \leq (1-\alpha) \overline{\Delta}^{2}$, then
\begin{equation}
\label{eqn:lower_bound__K=2}
\inf_{\hat{\eta}} \sup_{\|\vmu\| \geq \Delta/2} \Prob(\hat{\eta} \neq \eta) \geq 1-c n^{-1},
\end{equation}
where $c > 0$ is a constant depending only on $\alpha$ and the infimum is over all possible estimators $\hat\eta$ for $\eta=\{\eta_i\}_{i=1}^n \in \{\pm 1\}^{n}$.
\end{lem}

\begin{rem}
Our Lemma~\ref{lem:lower_bound__K=2} is stronger than the exact recovery notation in~\cite{Ndaoud2019} and the probability of wrong recovery lower bound in~\eqref{eqn:lower_bound__K=2} does not follows from the lower bound therein for the expected Hamming distance loss in the symmetric two-component Gaussian mixture model. Moreover, complementing the upper bound in Corollary~\ref{cor:SDP_bound}, the lower bound is sharply optimal in the sense that the probability of wrong recovery is arbitrarily close to one (rather than just bounded away from zero) if the separation signal size $\Delta^{2}$ is below the cutoff value $\overline{\Delta}^{2}$. \qed
\end{rem}

\begin{proof}[Proof of Lemma~\ref{lem:lower_bound__K=2}]
To prove the lower bound, we follow the same setup as in the lower bound proof in~\cite{Ndaoud2019} by placing a $N(0,\kappa_{n}^{2} I_p)$ prior on $\vmu$ and an independent Rademacher prior on $\eta$. Note that algorithm that maximizing the probability of reconstructing labels correctly is the maximum a posterior (MAP) estimator $\widetilde\eta =\argmax_{\eta} p(\eta\,|\,\vX)$. Since the prior label assignment is uniform, MAP is in particular equivalent to maximum (integrated) likelihood estimator (MLE) after integrating out $\vmu$, i.e., $\widetilde\eta = \argmax_{\eta} L(\eta \,|\,\vX)$ where $L(\eta \,|\,\vX) = p(\vX\,|\,\eta)$ is viewed as a function of $\eta$. Specifically, the maximum (integrated) likelihood function can be computed as follows
\begin{align*}
L(\eta \,|\,\vX) &= \int_{\R^p} \prod_{i=1}^n p(\vX_i\,|\,\vmu,\eta_i) \,p(\vmu)\, \rd\vmu\\
&\propto \int_{\R^p} \exp\Big\{-\frac{1}{2\sigma^2} \sum_{i=1}^n\|\vX_i-\eta_i\vmu\|^2 -\frac{1}{2\kappa_{n}^2}\|\vmu\|^2\Big\}\, \rd\vmu\\
&\propto \exp\Big\{\frac{1}{2} \Big(\frac{n}{\sigma^2}+\frac{1}{\kappa_{n}^2}\Big)^{-1} \Big\|\frac{1}{\sigma^2}\sum_{i=1}^n \eta_i \vX_i\Big\|^2\Big\}.
\end{align*}
We can see from the last expression that the MLE fails if there exists some $i\in[n]$ such that $\big\|\eta_i\vX_i+\sum_{j\neq i}^n \eta_j \vX_j\big\|^2 < \big\|-\eta_i\vX_i+\sum_{j\neq i}^n \eta_j \vX_j\big\|^2$, or equivalently,
$\big\langle \eta_i\vX_i,\, \sum_{j\neq i}^n \eta_j \vX_j\big\rangle <0$.
Therefore,
\begin{align}\label{Eqn:reduction}
\inf_{\widehat\eta}\Prob(\widehat{\eta}\neq \eta) = \Prob(\widetilde\eta \neq \eta) \geq  \Prob\Big(\exists i\in[n],\mbox{\, such that\, } \big\langle \eta_i\vX_i,\, \sum_{j\neq i}^n \eta_j \vX_j\big\rangle <0\Big).
\end{align}
Without loss of generality, we assume $\sigma=1$. Let $\overline{\vepsilon}_{n} = n^{-1} \sum_{i=1}^{n} \vepsilon_i$ be the sample average of the noise. Since $(\eta_1 \vepsilon_1,\dots,\eta_n \vepsilon_n)$ has the same joint distribution as $(\vepsilon_1,\dots,\vepsilon_n)$, we can write
\begin{align*}
 &\,\frac{1}{n-1} \big\langle \eta_i\vX_i,\, \sum_{j\neq i}^n \eta_j \vX_j\big\rangle = \big \langle \vmu + \eta_i \vepsilon_i, \, \mu+ \frac{1}{n-1} \sum_{j \neq i} \eta_i \vepsilon_i \big \rangle \stackrel{d}{=} \Big\langle \vmu + \vepsilon_i, \, \vmu + \frac{n}{n-1}\bar\vepsilon_{n} - \frac{1}{n-1}{\vepsilon_i}\Big\rangle\\
 & = \underbrace{\Big\langle\vepsilon_i-\overline{\vepsilon}_n, \, \vmu + \frac{n}{n-1}\overline{\vepsilon}_n\Big\rangle + \|\vmu\|^2}_{=:R_{i,1}} +\underbrace{\frac{n}{n-1}\|\overline{\vepsilon}_n\|^2 -\frac{1}{n-1}\|\vepsilon_i\|^2}_{=:R_{i,2}}+\underbrace{\frac{2n-1}{n-1}\langle \vmu, \overline{\vepsilon}_n\rangle - \frac{1}{n-1}\langle \vmu,\vepsilon_i\rangle}_{=:R_{i,3}}.
\end{align*}

\noindent{\bf Bound $R_{i,3}$.} Let $\beta_n > 0$ and 
\begin{eqnarray*}
\cB_1 &=& \big\{ \sqrt{n} |\langle \vmu, \overline{\vepsilon}_n \rangle | \leq \sqrt{n-1} \beta_n \|\vmu\|^2, \; \max_{i \in [n]} |\langle \vmu, \vepsilon_i \rangle | \leq \sqrt{n-1} \beta_n \|\vmu\|^2 \big\}, \\
\widetilde{\cB}_1 &=& \big\{ \sqrt{n} |\langle \vmu, \overline{\vepsilon}_n \rangle | \leq \sqrt{n-1} \beta_n \|\vmu\|, \; \max_{i \in [n]} |\langle \vmu, \vepsilon_i \rangle | \leq \sqrt{n-1} \beta_n \|\vmu\| \big\}.
\end{eqnarray*}
By the standard tail inequality of the Gaussian random variable and union bound, we have $\Prob(\cB_1^c) \leq \min\{ 1, n \exp(-c n \beta_n^2 \|\vmu\|^2) \}$ and $\Prob(\widetilde{\cB}_1^c) \leq \min\{1, n \exp(-c n \beta_n^2) \}$ for some universal constant $c > 0$. In addition, we have $\max_{i \in [n]} |R_{i,3}| \leq 3\beta_n\|\vmu\|^2$ on the event $\cB_1$ and $\max_{i \in [n]} |R_{i,3}| \leq 3\beta_n\|\vmu\|$ on the event $\widetilde\cB_1$.

\vspace{1em}
\noindent{\bf Bound $R_{i,2}$.} By tail inequalities of the chi-square random variable in Lemma~\ref{lem:tail_bound_chisquare}, we have for any $t > 0$ and $i \in [n]$,
\[
\Prob(\big| \|\vepsilon_i\|^2 - p \big| \geq 2\sqrt{pt} + 2t) \leq 2 e^{-t} \quad \mbox{and} \quad \Prob(\big| n \|\overline{\vepsilon}\|^2 -p \big| \geq 2 \sqrt{pt} + 2t) \leq 2 e^{-t}.
\]
Thus we have $\Prob(\cB_2^c) \leq 4 n^{-1}$, where
\[
\cB_2 = \big\{ \big| n \|\overline\vepsilon_n\|^2 - p \big| \leq 2 \sqrt{p \log{n}} + 2 \log{n}, \; \max_{i \in [n]} \big| \|\vepsilon_i\|^2 - p \big| \leq 2 \sqrt{2 p \log{n}} + 4 \log{n} \big\}.
\]
On the event $\cB_2$, we have $\max_{i \in [n]} |R_{i,2}| \leq 6 (\sqrt{p \log{n}} + \log{n})/(n-1)$.

\vspace{1em}
\noindent{\bf Analyze $R_{i,1}$.} From elementary calculations, we have that the conditional joint distribution of $U_i:\,=\Big\langle\vepsilon_i-\overline\vepsilon_n, \, \vmu + \frac{n}{n-1}\overline\vepsilon_n\Big\rangle$, $i=1,\ldots,n$, given $\overline\vepsilon_n$ is
\begin{align*}
\left. 
\begin{pmatrix}
U_1 \\ U_2 \\ \vdots \\ U_n
\end{pmatrix} \, \right|\,  \overline\vepsilon_n \sim N \left(
\begin{pmatrix}
0 \\ 0 \\ \vdots \\ 0
\end{pmatrix},\, \Big \|\vmu + \frac{n}{n-1}\overline\vepsilon_n\Big\|^2
\begin{pmatrix}
1-n^{-1} & -n^{-1} & \cdots &  -n^{-1}\\
-n^{-1} & 1-n^{-1} & \cdots &  -n^{-1}\\
\vdots & \vdots & \ddots & \vdots\\
-n^{-1} & -n^{-1} & \cdots &  1-n^{-1}\\
\end{pmatrix}\right).
\end{align*}
Conditioning on $\overline\vepsilon_n$, let $\{Z_i\}_{i=1}^n$ be i.i.d. $N\big(0, (1-n^{-1})\big \|\vmu + \frac{n}{n-1}\overline\vepsilon_n\big\|^2\big)$ random variables. Since $\mathbb E(U_i^2\,|\,\overline\vepsilon_n) = \mathbb E(Z_i^2\,|\,\overline\vepsilon_n)$ and $\mathbb E((U_i-U_j)^2\,|\,\overline\vepsilon_n) \geq  \mathbb E((Z_i-Z_j)^2\,|\,\overline\vepsilon_n)$ for $i,j\in[n]$, by Slepian's inequality (cf. Theorem 7.2.9 in~\cite{Vershynin2018_Cambridge}) we have
\begin{align*}
\Prob \Big(\max_{i\in [n]} U_i > t \, \Big| \, \overline\vepsilon_n\Big) \geq \Prob \Big(\max_{i\in [n]} Z_i > t\,\Big|\,\overline\vepsilon_n\Big) = 1-\Big(1-\Prob\big(Z_1 >t\,\big|\,\overline\vepsilon_n\big)\Big)^n,\quad t\in\mathbb R.
\end{align*}
Combining the previous three terms with~\eqref{Eqn:reduction}, we obtain that
\begin{align}
\nonumber
\Prob(\widehat{\eta}\neq \eta) & \geq 1- \Prob\Big(R_{i,1} + R_{i,2} + R_{i,3} \geq 0, \; \forall i\in[n] \Big) \\\nonumber
& = 1 - \Prob\Big(-U_{i} \leq \|\vmu\|^2 + R_{i,2} + R_{i,3}, \; \forall i\in[n] \Big) \\ \label{eqn:master_lower_bound_two-components}
& \geq \left\{
\begin{array}{cc}
1 - \Prob\Big(-U_{i} \leq (1+3\beta_n) \|\vmu\|^2 + {6 (\sqrt{p \log{n}} + \log{n}) \over n-1}, \; \forall i\in[n] \Big) &  \mbox{on } \cB_1 \cap \cB_2 \\
1 - \Prob\Big(-U_{i} \leq \|\vmu\|^2 + 3\beta_n\|\vmu\| + {6 (\sqrt{p \log{n}} + \log{n}) \over n-1}, \; \forall i\in[n] \Big) &  \mbox{on } \widetilde\cB_1 \cap \cB_2 \\
\end{array}
\right. .
\end{align}
For $\gamma_n >0$, we define
\[
\cB_3 = \big\{ \big| \|\vmu\|^2 - \Delta^2/4 \big| \leq \Delta^2 \gamma_{n}/4\big\}.
\]
With the prior distribution $\vmu \sim N(0, 4^{-1}\kappa_n^2 I_p)$ where $\kappa_n^2 = {\Delta^2 \over 4p (1-\nu_n)}$ and $\nu_n = \sqrt{n \Delta^2 \over 4p \log^2{n}}$, it follows from the proof of Theorem 5 in~\cite{Ndaoud2019} (cf. equation (28)) that $\Prob(\cB_3^c) \leq 2\exp(-p \gamma_n^2 / 32)$, provided $4\nu_n \leq \gamma_n \leq 1$. Moreover, using the lower tail bound of the chi-square random variable in Lemma~\ref{lem:tail_bound_chisquare}, we have $\Prob(\cB_4^c) \leq \exp(-p \theta_n^2 / 4)$, where 
\[
\cB_4 = \big\{\|\overline\vepsilon_n\|^2 \geq {p \over n} (1-\theta_n)\big\}.
\]
To analyze the right-hand side of~\eqref{eqn:master_lower_bound_two-components}, we first consider the higher dimensional case where $p \geq \log^{2}{n}$. In such regime, we divide further into three cases depending on the separation signal size as following.

\noindent{\bf Medium signal size case: ${2\log^{3/2}{n} \over \sqrt{n}} < \Delta < 2\sqrt{p \log{n} \over n}$.} Since $-(U_1,\dots,U_n)$ has the same joint distribution of $(U_1,\dots,U_n)$ given $\overline\vepsilon_n$, we can bound on $\cB_1 \cap \cB_2$,
\begin{align*}
& \Prob\Big(-U_{i} \leq (1+3\beta_n) \|\vmu\|^2 + {6 (\sqrt{p \log{n}} + \log{n}) \over n-1}, \; \forall i\in[n] \Big) \\
= & \Prob\Big(\max_{i \in [n]} U_{i} \leq (1+3\beta_n) \|\vmu\|^2 + {6 (\sqrt{p \log{n}} + \log{n}) \over n-1} \Big) \\
\leq & \E\Big[ 1 - \Prob\Big( Z_1 > (1+3\beta_n) \|\vmu\|^2 + {6 (\sqrt{p \log{n}} + \log{n}) \over n-1} |\; \overline\vepsilon_n \Big) \Big]^n.
\end{align*}
Let $Z \sim N(0, 1)$ be the standard Gaussian random variable. Thus, on the event $\bigcap_{i=1}^{4} \cB_i$, we have
\begin{align*}
& \Prob\Big( Z_1 > (1+3\beta_n) \|\vmu\|^2 + {6 (\sqrt{p \log{n}} + \log{n}) \over n-1} |\; \overline\vepsilon_n \Big) \\
= & \Prob\left( Z > {(1+3\beta_n) \|\vmu\|^2 + {6 (\sqrt{p \log{n}} + \log{n}) \over n-1}  \over \sqrt{1-{1\over n}} \sqrt{\|\vmu\|^2 + {n^2 \over (n-1)^2} \|\overline{\vepsilon}_n\|^2 + {2 n \over n-1} \langle \vmu, \overline\vepsilon_n \rangle} } |\; \overline\vepsilon_n \right) \\
\geq & \Prob\left( Z > {(1+3\beta_n) (1+\gamma_n) \Delta^2 + {24 (\sqrt{p \log{n}} + \log{n}) \over n-1}  \over  \sqrt{4(1-{1\over n}-2\beta_n) (1-\gamma_n) \Delta^2 + {16p \over n-1} (1-\theta_n) } } \right) = \Phi^c(V_n),
\end{align*}
where
\[
V_n = {(1+3\beta_n) (1+\gamma_n) \Delta^2 + {24 (\sqrt{p \log{n}} + \log{n}) \over n-1}  \over \sqrt{4(1-{1\over n}-2\beta_n) (1-\gamma_n) \Delta^2 + {16p \over n-1} (1-\theta_n)} }.
\]
and $\Phi^c(t) = \Prob(Z\geq t)$. 
Combining all pieces together, we obtain that
\begin{equation}
\label{eqn:wrong_recovery_prob_high_dim_medium_signal}
\Prob(\widetilde{\eta}\neq \eta) \geq 1 - [1- (1-r_n) \Phi^c(V_n)]^{n} - r_n
\end{equation}
provided $4 \nu_n \leq \gamma_n \leq 1$, where
\[
r_n = \min\{ 1, n \exp(-c n \beta_n^2 (1-\gamma_n^2) \Delta^2) \} + 4n^{-1} + 2\exp(-p \gamma_n^2 / 32) + \exp(-p \theta_n^2 / 4).
\]
Note that $\Delta < 2\sqrt{p \log{n} \over n}$, implying $\nu_n^2 \leq {1 \over \log{n}}$. We choose
\[
\beta_n^2 = {1 \over \log{n}}, \quad \gamma_n^2 = {16 \over \log{n}}, \quad \theta_n^2 = {1 \over \log{n}}.
\]
Since $p \geq \log^2{n}$ and $\Delta > {2\log^{3/2}{n} \over \sqrt{n}}$, we have
\[
\Prob(\widetilde{\eta}\neq \eta) \geq 1 - [1- (1-cn^{-1}) \Phi^c(V_n)]^{n} - cn^{-1}.
\]
By using the fact that $\log \Phi^c(t)\sim -t^2/2$ as $t\to\infty$, we can conclude that as long as $V_n \leq \sqrt{2(1-\delta)\log n}$ for any $\delta \in (0, 1)$, then for $n \geq 2 c$, we have
\[
\Prob(\widetilde{\eta}\neq \eta) \geq 1 - \Big(1 - {1-cn^{-1} \over n^{1-\delta}} \Big)^{n} - cn^{-1} \geq 1 - e^{-n^{\delta}/2} - c n^{-1} \geq 1 - c'n^{-1},
\]
where $c'$ is a constant depending on $\delta$. The condition $V_n \leq \sqrt{2(1-\delta)\log n}$ is implied by $\Delta^{2} \leq (1-\alpha) \overline{\Delta}^{2}$ for some $\delta := \delta(\alpha)$ and by inverting the function $x \mapsto {x / \sqrt{4x + 16p/n}}$ for $x > 0$.

\vspace{1em}
\noindent{\bf Low signal size case: $\Delta \leq {2\log^{3/2}{n} \over \sqrt{n}}$.} The argument is similar to the medium signal size case, so we only sketch the proof. On the event $\widetilde{\cB}_1 \bigcap \cB_2 \bigcap \cB_3 \bigcap \cB_4$, we have
\begin{align*}
& \Prob\Big(\max_{i \in [n]} U_{i} \leq \|\vmu\|^2 + 3\beta_n \|\vmu\| + {6 (\sqrt{p \log{n}} + \log{n}) \over n-1} \Big) \\
\leq & \E\Big[ 1 - \Prob\Big( Z_1 > \|\vmu\|^2 + 3\beta_n \|\vmu\| + {6 (\sqrt{p \log{n}} + \log{n}) \over n-1} |\; \overline\vepsilon_n \Big) \Big]^n
\end{align*}
and
\begin{align*}
& \Prob\Big( Z_1 > \|\vmu\|^2 + 3\beta_n \|\vmu\| + {6 (\sqrt{p \log{n}} + \log{n}) \over n-1} |\; \overline\vepsilon_n \Big) \\
\geq & \Prob\left( Z > {\|\vmu\|^2 + 3\beta_n \|\vmu\|+ {6 (\sqrt{p \log{n}} + \log{n}) \over n-1}  \over \sqrt{1-{1\over n}} \sqrt{\|\vmu\|^2 + {n^2 \over (n-1)^2} \|\overline{\vepsilon}_n\|^2 -2 \beta_n \|\vmu\| } } |\; \overline\vepsilon_n \right) \\
\geq & \Prob\left( Z > {(1+\gamma_n) \Delta^2 + 6 \beta_n \sqrt{1+\gamma_n} \Delta + {24 (\sqrt{p \log{n}} + \log{n}) \over n-1}  \over  \sqrt{4(1-{1\over n}) (1-\gamma_n) \Delta^2 + {16p \over n-1} (1-\theta_n) - 16 \beta_n \sqrt{1-\gamma_n} \Delta} } \right) = \Phi^c(V_n),
\end{align*}
where
\[
V_n = {(1+\gamma_n) \Delta^2 + 6 \beta_n \sqrt{1+\gamma_n} \Delta + {24 (\sqrt{p \log{n}} + \log{n}) \over n-1}  \over  \sqrt{4(1-{1\over n}) (1-\gamma_n) \Delta^2 + {16p \over n-1} (1-\theta_n) - 16 \beta_n \sqrt{1-\gamma_n} \Delta} }.
\]
Combining all pieces together, we obtain that
\begin{align*}
\Prob(\widetilde{\eta}\neq \eta) \geq 1 - [1-(1-r_n)\Phi^c(V_n)]^{n} - r_n,
\end{align*}
provided $4 \nu_n \leq \gamma_n \leq 1$, where
\[
r_n = \min\{ 1, n \exp(-c n \beta_n^2) \} + 4n^{-1} + 2\exp(-p \gamma_n^2 / 32) + \exp(-p \theta_n^2 / 4).
\]
Now we choose
\[
\beta_n^2 = {\log^2{n} \over n}, \quad \gamma_n^2 = {16 \over \log{n}}, \quad \theta_n^2 = {1 \over \log{n}}.
\]
If ${2\log^2{n} \over n} < \Delta \leq {2\log^{3/2}{n} \over \sqrt{n}}$, then $\beta_n \Delta \leq p/n$ (recall $p \geq \log^2{n}$) and there exists a sequence $\xi_n \to 0$ as $n \to \infty$ such that
\[
V_n \leq (1+\xi_n) (\Delta/2 + 3 \beta_n + \xi_n) = o(1),
\]
which implies that $\Prob(\widetilde{\eta}\neq \eta) \geq 1 - c n^{-1}$. If $\Delta \leq {2\log^2{n} \over n}$, then
\[
V_n \leq (1+\xi_{n}) {3 \Delta \beta_n \over 2\sqrt{p / n}} = o(1)
\]
and $\Prob(\widetilde{\eta}\neq \eta) \geq 1 - c n^{-1}$.

\vspace{1em}
\noindent{\bf High signal size case: $2\sqrt{p \log{n} \over n} \leq \Delta \leq \sqrt{1-\alpha} \; \overline{\Delta}$.} Note that in this regime, we have $p/n = o(\Delta^2)$ and $p = O(n)$. Then the sharp threshold $\overline\Delta^2 = 8 (1+o(1)) \log{n}$, which is asymptotically independent of $p$. Thus we place an (essentially one-dimensional) point mass prior on $\vmu$ at $(\Delta/2,0,\ldots,0)^T \in\mathbb R^p$. A similar calculation yields
\begin{align*}
&L(\eta\,|\,\vX) \propto \exp\Big\{\frac{1}{\sigma^2} \Big\langle \vmu,\,\sum_{i=1}^n\eta_i\vX_i\Big\rangle\Big\},\quad\mbox{and}\\
&\Prob(\widetilde\eta\neq \eta) \geq \Prob\big(\exists i\in[n],\mbox{\, such that\, } \big\langle \vmu,\, \eta_i\vX_i\big\rangle <0\big).
\end{align*}
Since $\big\{\langle \vmu,\,\eta_i\vX_i\rangle\big\}_{i=1}^n$ are i.i.d. random variables with
\begin{align*}
&\Prob\big(\langle \vmu,\, \eta_i\vX_i\big\rangle \geq 0\big) = \Prob\big( \|\vmu\|^2 + \langle \vmu, \eta_i\vepsilon_i\rangle\geq 0\big)
= \Prob(Z \geq - \|\vmu\|) = 1 - \Phi^c(\Delta/2),
\end{align*}
we have 
\begin{align*}
\Prob(\widetilde\eta\neq \eta) \geq 1- \big(1- \Phi^c(\Delta/2)\big)^n \geq 1-\Big(1-\frac{1}{n^{1-\delta}}\Big)^n \geq 1 - e^{-n^\delta} \geq 1-cn^{-1},
\end{align*}
when $\Delta^2 \leq 4\,(2-\delta)\log n \leq (1-\alpha) \overline{\Delta}^2$ for some $\delta$ depending only on $\alpha$. Here the constant $c$ depending only on $\delta$ (and thus only on $\alpha$).

\vspace{1em}
Finally, we deal with the lower dimensional case where $p < \log^{2}{n}$. In such regime, we also have $\overline{\Delta}^{2} = 8 (1+o(1)) \log{n}$. Following the same argument as in the {\bf high signal size case} under $p \geq \log^2{n}$, we conclude that $\Prob(\widetilde\eta\neq \eta) \geq 1 - c n^{-1}$.
\end{proof}

\section{Discussions}
\label{sec:discussions}

In this paper, we characterized the information-theoretic sharp threshold for exact recovery of Gaussian mixture models. There are still some interesting open questions, which we list below.

\vspace{1em}
\noindent {\bf General noise covariance matrix.} The SDP relaxation in~~\eqref{Eqn:primal_SDP_PW} does not require to know the noise variance $\sigma^{2}$ only in the spherical Gaussian case, i.e., the noise $\vepsilon_{i}$ has i.i.d. $N(0, \sigma^{2} I_{p})$ distribution. Consider the general covariance matrix case $\vepsilon_{i} \sim N(0, \Sigma)$ when $\Sigma$ is not necessarily spherical. If $\Sigma$ is known, then we can first apply the transform $\Sigma^{-1/2} X_{i}$ to make the noise spherical and the sharp threshold in~\eqref{eqn:separation_lower_bound_ChenYang} holds and reads in terms of the minimal Mahalanobis distance
\[
\widetilde\Delta^2=\min_{1 \leq j < k \leq K} d_{\Sigma}^{2}(\mu_{j},\mu_{k}) = 4 \left(1+\sqrt{1+{Kp \over n \log{n}}}\right) \log{n},
\]
where $d_{\Sigma}^{2}(\mu_{j},\mu_{k}) = (\mu_{j}-\mu_{k})^{T} \Sigma^{-1} (\mu_{j}-\mu_{k})$.
If $\Sigma$ is unknown, \cite{Giraud2015} showed that in the $K=2$ case the misclassification probability of the Bayes classifier decays exponentially fast in the Mahalanobis distance $d_{\Sigma}^{2}(\mu_{1},\mu_{2}) = (\mu_{1}-\mu_{2})^{T} \Sigma^{-1} (\mu_{1}-\mu_{2})$ rather than ${\Delta^{2} \over \sigma^{2}} = {\Delta^{2} \over \|\Sigma\|_{op}}$. Thus we conjecture that:
\begin{quote}
\emph{there is a sharp threshold for exact recovery under the general unknown covariance matrix given by $\widetilde\Delta^2$ above.
}
\end{quote}

\vspace{1em}
\noindent {\bf Average-case algorithmic hardness in multiple clusters.} Both our upper and lower bounds for exact recovery in Corollary~\ref{cor:SDP_bound} and Theorem~\ref{thm:SDP_low_bound} require the number of clusters $K = O(\log{n})$. We argue that this condition is likely to be necessary for achieving the sharp threshold of exact recovery. Consider the balanced spherical Gaussian mixture model with common noise variance and multiple communities for $K \geq 3$. It is shown in \cite{BanksMooreVershyninVerzelenXu2018} that: (i) detection and partial (i.e., correlated) recovery are information-theoretically possible if $\rho > 2 \sqrt{p K \log{K} \over n} + 2 \log{K}$; (ii) detection and partial recovery are impossible if $\rho < \sqrt{2p (K-1)\log(K-1) \over n}$, where $\rho$ is the squared signal-to-noise ratio in the Gaussian mixture model (an equivalent quantity of $\Delta^{2} / \sigma^{2}$ in our notation). In contrast, it is also known from~\cite{BanksMooreVershyninVerzelenXu2018,WuXu2018} that spectral methods have correlated recovery with the true community labels if and only if $\rho >  \sqrt{p \over n} (K-1)$. The phase transition of spectral methods is a direct consequence of the BBP phase transition in the random matrix theory~\cite{baik2005,Paul2007_SS}.  Thus for fixed $K$, there is no gap (modulo constants) between computation and information theoretic thresholds. In addition, a sufficient condition for partial recovery of the same SDP as in our paper is given by ${\Delta^{2} \over \sigma^{2}} \gtrsim (1+\sqrt{p/n})K$ in~\cite{GiraudVerzelen2018}. Based on evidence from statistical physics, it is conjectured by~\cite{lesieur2016phase} (and remains as an open problem) that the computational threshold coincides with the spectral methods for partial recovery for large $K$, thus suggesting there is a computationally hard regime where no polynomial time algorithm can attain the information-theoretic threshold when $K \to \infty$.

Now turning into exact recovery. Recall that our result shows that the information-theoretical threshold is
\[
{\overline{\Delta}^{2} \over \sigma^{2}} = 4 \left(\log{n}+\sqrt{\log^{2}{n}+{Kp \log{n} \over n}}\right),
\]
which is achieved by an SDP when $K \lesssim \log(n) / \log\log(n)$. Thus, in such growth region of $K$, there is no computational hardness for exact recovery, which is a similar scenario in the partial recovery case when $K = O(1)$. Note that the threshold ${\overline{\Delta}^{2} / \sigma^{2}}$ is larger (modulo constants) than the partial recovery sufficient condition for the spectral methods and the SDP, which in turn is strictly larger than its necessary condition (i.e., information-theoretic threshold) as $K \to \infty$. Hence, we propose the following conjecture:
\begin{quote}
\emph{for $K \gg \log{n}$, there is no polynomial time algorithm can achieve the average-case exact recovery information-theoretic threshold.}
\end{quote}
If this conjecture is true, then our current regime $K \lesssim \log(n) / \log\log(n)$ where the SDP achieves the information-theoretic limit is sharp, i.e., $K \asymp \log(n)$ would be an algorithmic hardness for exact recovery. The conjecture also implies that transition of hardness of clustering Gaussian mixture models from partial recovery and exact recovery is from $O(1)$ and $O(\log{n})$, respectively.

\vspace{1em}
\noindent {\bf Unbalanced communities.} Corollary~\ref{cor:SDP_bound} and Theorem~\ref{thm:SDP_low_bound} together imply that in the equal cluster size case when $n_1=n_2=\cdots=n_K=\frac{n}{K}$, the SDP relaxation~\eqref{Eqn:primal_SDP_PW} for the $K$-means is minimax-optimal in the sense that sharp phase transition of the probability of wrong recovery from zero to one occurs at the critical threshold given by the $\overline \Delta^2$ in~\eqref{eqn:separation_lower_bound_ChenYang}. It remains an interesting open question whether the separation gap $\overline \Delta$ is sharp when cluster sizes are unbalanced.

\section{Proof of key lemmas}

\subsection{Proof of Lemma~\ref{lem:separation_bound}}
Without loss of generality, we may assume $\sigma=1$. Denote $\mvtheta = \vmu_{k}-\vmu_{l}$ and define the event $\mathcal A=\bigcap_{k,l,i} \mathcal A^{(i)}_{kl}$, where
\begin{align*}
\mathcal A^{(i)}_{kl} =\Big\{\|\vX_i - \overline \vX_l\|^2 -  \|\vX_i - \overline\vX_k\|^2 \geq \frac{n_k+n_l}{n_kn_l}p + \beta\,\|\mvtheta\|^2 - r_{kl} \Big\},
\end{align*}
with the index $(k,l,i)$ ranging over all distinct pairs $(k,l)\in[K]^2$ and all $i\in G^\ast_k$ and 
\[
r_{kl} = 2 \sqrt{2\log(nK) \over n_{l}} \|\mvtheta\| + 2 \frac{n_k+n_l}{n_kn_l} \sqrt{2p\log(nK)} +{4\over n_{k}} \log(nK).
\]
Recall that $\vX_i = \vmu_k+\vepsilon_i$ for each $i\in G^\ast_k$ and $k\in [K]$. We can write 
\begin{align*}
& \|\vX_i - \overline{\vX}_l\|^2 -  \|\vX_i - \overline{\vX}_k\|^2 = \|\mvtheta+\vepsilon_{i}-\overline{\vepsilon}_{l}\|^{2} - \|\vepsilon_{i}-\overline{\vepsilon}_{k}\|^{2} = \langle \mvtheta - \overline{\vepsilon}_{l} + \overline{\vepsilon}_{k}, \mvtheta - \overline{\vepsilon}_{l} + 2\vepsilon_{i} - \overline{\vepsilon}_{k} \rangle \\
& \qquad = \|\mvtheta\|^{2} + \|\overline{\vepsilon}_{l}\|^{2} + {1\over n_{k}}\left(2-{1\over n_{k}}\right)\|\vepsilon_{i}\|^{2} - \left( {n_{k}-1 \over n_{k}} \right)^{2} \|\overline{\vepsilon}_{k\setminus\{i\}}\|^{2} - 2 \langle \mvtheta, \, \overline{\vepsilon}_{l} \rangle \\
& \qquad \qquad + 2 \left\langle \mvtheta - \overline{\vepsilon}_{l}+ \left( {n_{k}-1 \over n_{k}} \right)^{2} \overline{\vepsilon}_{k\setminus\{i\}}, \, \vepsilon_{i} \right\rangle,
\end{align*}
where $\overline{\vepsilon}_{k\setminus\{i\}} = (n_{k}-1)^{-1} \sum_{j \in G^{\ast}_{k} \setminus\{i\}} \vepsilon_{j}$. Set $\zeta_{n} = 2\log(nK)$ and define 
\begin{align}
\mathcal B_{kl}^{(i)} =\Big\{& \|\overline{\vepsilon}_{l}\|^{2} \geq n_{l}^{-1} (p-2\sqrt{p\zeta_{n}}), \,\, \|\vepsilon_{i}\|^{2} \geq p - 2\sqrt{p \zeta_{n}}, \notag \\
& \|\overline{\vepsilon}_{k\setminus\{i\}}\|^{2} \leq (n_{k}-1)^{-1} (p + 2\sqrt{p \zeta_{n}} + 2\zeta_{n}), \,\, \langle \mvtheta, \overline{\vepsilon}_{l} \rangle \leq \sqrt{2 n_{l}^{-1} \zeta_{n}} \|\mvtheta\|, \notag \\
& \|\overline{\vepsilon}_{l} - ({1-n_{k}^{-1}})^{2}\overline{\vepsilon}_{k\setminus\{i\}}\|^{2} \leq (n_{l}^{-1}+n_{k}^{-1}) (p + 2\sqrt{p \zeta_{n}} + 2\zeta_{n}), \label{eqn:Z2Z3a} \\
& \langle \mvtheta, ({1-n_{k}^{-1}})^{2} \overline{\vepsilon}_{k\setminus\{i\}}-\overline{\vepsilon}_{l} \rangle \leq \sqrt{2 (n_l^{-1}+n_k^{-1})\zeta_n} \|\mvtheta\| \Big\}.\label{eqn:Z2Z3b}
\end{align}
Note that $\vepsilon_{i}, \overline{\vepsilon}_{l}$, and $\overline{\vepsilon}_{k\setminus\{i\}}$ are mutually independent. Thus conditional on $\overline{\vepsilon}_{l}$ and $\overline{\vepsilon}_{k\setminus\{i\}}$, we have 
\[
\langle \mvtheta - \overline{\vepsilon}_{l}+ ({1-n_{k}^{-1}})^{2} \overline{\vepsilon}_{k\setminus\{i\}}, \, -\vepsilon_{i} \rangle \sim N \Big( 0, \Big\| \mvtheta - \overline{\vepsilon}_{l}+ ({1-n_{k}^{-1}})^{2} \overline{\vepsilon}_{k\setminus\{i\}} \Big\|^{2} \Big).
\]
Then, on the event where~\eqref{eqn:Z2Z3a} and~\eqref{eqn:Z2Z3b} hold, we can bound 
\begin{align*}
& U^\ast := \Prob \Big( 2 \langle \mvtheta - \overline{\vepsilon}_{l}+ ({1-n_{k}^{-1}})^{2} \overline{\vepsilon}_{k\setminus\{i\}}, \, -\vepsilon_{i} \rangle \geq (1-\beta) \|\mvtheta\|^{2}  \, \Big| \,\overline{\vepsilon}_{l}, \overline{\vepsilon}_{k\setminus\{i\}} \Big) \\
& \qquad = \Phi^{c} \left( {(1-\beta) \|\mvtheta\|^{2} \over 2 \sqrt{\| \mvtheta \|^{2} + 2 \langle \mvtheta, ({1-n_{k}^{-1}})^{2} \overline{\vepsilon}_{k\setminus\{i\}}-\overline{\vepsilon}_{l} \rangle + \|\overline{\vepsilon}_{l} - ({1-n_{k}^{-1}})^{2}\overline{\vepsilon}_{k\setminus\{i\}}\|^{2} } } \right) \\
& \qquad \leq \Phi^{c} \left( {(1-\beta) \|\mvtheta\|^{2} \over 2 \sqrt{\| \mvtheta \|^{2} + 2 \sqrt{2(n_l^{-1}+n_k^{-1})\zeta_n} \|\mvtheta\| + (n_{l}^{-1}+n_{k}^{-1}) (p + 2\sqrt{p \zeta_{n}} + 2\zeta_{n}) } } \right),
\end{align*}
where $\Phi^{c}(t)$ denotes the tail probability $\mathbb P(Z\geq t)$ for a standard normal random variable $Z$. Note that $n_{l}^{-1}+n_{k}^{-1} \leq 2 m^{-1}$. Under the separation condition~\eqref{eqn:separation_bound} on the Gaussian centers, we have $\|\mvtheta\|^2 \geq 8 \log{n}$ and 
\[
2 \sqrt{2(n_l^{-1}+n_k^{-1})\zeta_n} \|\mvtheta\| \leq {2 \over \sqrt{m}} \|\mvtheta\|^2.
\]
Thus, on events~\eqref{eqn:Z2Z3a} and~\eqref{eqn:Z2Z3b}, we have
\[
U^\ast \leq \Phi^{c} \left( {(1-\beta) \|\mvtheta\|^{2} \over 2 \sqrt{(1+ {2\over \sqrt{m}}) \| \mvtheta \|^{2} + (n_{l}^{-1}+n_{k}^{-1}) (p + 2\sqrt{p \zeta_{n}} + 2\zeta_{n}) } } \right).
\]
Now under the separation condition~\eqref{eqn:separation_bound} and noting that $(1+\delta)(1+{2\over\sqrt{m}}) \leq 1+2\delta$, we see that
\[
{(1-\beta)^{2}\over 8(1+\delta) \log{n}}  \|\mv\theta\|^{4}  - \left(1+ {2 \over \sqrt{m}}\right) \|\mvtheta\|^{2} - r_{1} \geq 0,
\]
where 
\[
r_{1} = p ( n_{l}^{-1} + n_{k}^{-1} ) + 2 (\sqrt{p\zeta_{n}}+\zeta_{n}) ( n_{l}^{-1} + n_{k}^{-1} ).
\]
Hence we get 
\begin{align*}
U^\ast \leq \Phi^{c}(\sqrt{2(1+\delta)\log{n}}) \leq n^{-(1+\delta)},
\end{align*}
where the second inequality follows from the standard Gaussian tail bound $\Phi^{c}(x) \leq e^{-x^{2}/2}$ for $x \geq 0$. In addition, applying the probability tail bounds for $\chi^{2}$ distributions in Lemma~\ref{lem:tail_bound_chisquare}, we have $\Prob({\cB_{kl}^{(i)}}^{c}) \leq 6 / (n^{2} K^{2})$.
Now putting pieces together, we have 
\begin{align*}
\Prob(\cA^{c}) &\leq \sum_{1\leq k \neq l \leq K} \sum_{i \in G_{k}^{*}} \Prob({\cA_{kl}^{(i)}}^{c} \cap \cB_{kl}^{(i)}) + \Prob({\cB_{kl}^{(i)}}^{c}) \\
&\leq \sum_{1\leq k \neq l \leq K} \sum_{i \in G_{k}^{*}} \E \left[ U^\ast \vone(\eqref{eqn:Z2Z3a},~\eqref{eqn:Z2Z3b} \mbox{ hold}) \right] + {6 \over n} \\
&\leq {K^{2} \over n^{\delta}} + {8 \over n}.
\end{align*}

\subsection{Proof of Lemma~\ref{lem:bound_Tv}}
Without loss of generality, we may assume $\sigma = 1$. Recall that the column sums and row sums of matrix $B_{G^\ast_kG^\ast_l}$ are denoted by $\vc^{(k,l)}=\big(c^{(k,l)}_j:\, j\in G^\ast_l\big)$ and $\vr^{(k,l)}=\big(r^{(k,l)}_i:\, i\in G^\ast_k\big)$, respectively. In addition, $t^{(k,l)}=\sum_{j\in G^\ast_l}c^{(k,l)}_j =\sum_{i\in G^\ast_k} r^{(k,l)}_i$ is the total sum, and the construction of $B^\sharp$ in~\eqref{eqn:B_form} can be written as $[B_{G_l^\ast G_k^\ast}^{\sharp}]_{ij}= r^{(k,l)}_i c^{(k,l)}_j/t^{(k,l)}$, for any distinct pair $(k,l)\in[K]^2$. Under this notation, for each $\vv \in\Gamma_K$, we may write 
\[
T(\vv) = \sum_{k=1}^K \sum_{l\neq k}\sum_{i\in G^\ast_k}\sum_{j\in G^\ast_l} \frac{r^{(k,l)}_i c^{(k,l)}_j}{t^{(k,l)}}\,v_i\,v_j = \sum_{k=1}^K \sum_{l\neq k}\bigg\{\frac1{t^{(k,l)}} \Big(\sum_{i\in G^\ast_k}v_i\, r^{(k,l)}_i\Big)\Big(\sum_{j\in G^\ast_l} v_j\, c^{(k,l)}_j\Big)\bigg\}.
\]
Using once again the property $\sum_{i \in G_{k}^{\ast}} v_{i} = 0$ for all $k \in [K]$, we can simplify 
\begin{align*}
\sum_{j\in G^\ast_l} v_j\, c^{(k,l)}_j = & \sum_{j\in G^\ast_l} v_j \left[ -{n_{l}+n_{k} \over 2 n_{l}} \lambda + {n_{k} \over 2} (\|\overline{\vX}_{k}\|^{2} - \|\overline{\vX}_{l}\|^{2}) + n_{k} \langle \vX_{j}, \overline{\vX}_{l} - \overline{\vX}_{k}\rangle \right] \\
= & n_{k} \langle \overline{\vX}_{l} - \overline{\vX}_{k}, \sum_{j\in G^\ast_l} v_j \vX_{j} \rangle \\
= & n_{k} \langle \overline{\vX}_{l} - \overline{\vX}_{k}, \sum_{j\in G^\ast_l} v_j \vepsilon_{j} \rangle \\
= & n_{k} \langle \vmu_{l}-\vmu_{k}+\overline\vepsilon_{l}-\overline\vepsilon_{k}, \sum_{j \in G^\ast_l} v_j\, \vepsilon_j \rangle.
\end{align*}
Similarly, $\sum_{i\in G^\ast_k}v_i\, r^{(k,l)}_i\ = n_{l} \langle  \vmu_{k}-\vmu_{l}+\overline\vepsilon_{k}-\overline\vepsilon_{l}, \sum_{i\in G^\ast_k} v_i\, \vepsilon_i \rangle$. Then 
\[
\sum_{i\in G^\ast_k} \sum_{j\in G^\ast_l} v_i v_j \, r^{(k,l)}_i c^{(k,l)}_j = -n_{k} n_{l} (T_{1,kl} + T_{2,kl} + T_{3,kl}),
\]
and 
\begin{align}\label{Eqn:T_decom}
T(\vv) =  -\sum_{k=1}^K \sum_{l\neq k}\bigg\{\frac{n_kn_l}{t^{(k,l)}}\, (T_{1,kl} + T_{2,kl} + T_{3,kl})\bigg\},
\end{align}
where 
\begin{align*}
T_{1,kl}(\vv) = & \bigg[\sum_{i\in G^\ast_k}  v_{i}  \langle \vmu_{k}-\vmu_{l}, \vepsilon_{i} \rangle  \bigg]\cdot \bigg[\sum_{j \in G^\ast_l} v_{j}\langle \vmu_{k}-\vmu_{l}, \vepsilon_{j} \rangle\bigg], \\
T_{2,kl}(\vv) = & \bigg[\sum_{i\in G^\ast_k} v_{i} \langle \overline\vepsilon_{k}-\overline\vepsilon_{l}, \vepsilon_{i} \rangle \bigg]\cdot \bigg[ \sum_{j \in G^\ast_l} v_{j} \langle \overline\vepsilon_{k}-\overline\vepsilon_{l},\vepsilon_{j} \rangle\bigg], \\
T_{3,kl}(\vv) = & \bigg[\sum_{i\in G^\ast_k}  v_{i} \langle \vmu_{k}-\vmu_{l}, \vepsilon_{i} \rangle \bigg] \cdot \bigg[ \sum_{j \in G^\ast_l}  v_{j} \langle \overline\vepsilon_{k}-\overline\vepsilon_{l},\vepsilon_{j} \rangle\bigg] \\
&\qquad\qquad\qquad +  \bigg[\sum_{i\in G^\ast_k}  v_{i} \langle \overline\vepsilon_{k}-\overline\vepsilon_{l}, \vepsilon_{i} \rangle \bigg] \cdot \bigg[ \sum_{j \in G^\ast_l}  v_{j} \langle \vmu_{k}-\vmu_{l},\vepsilon_{j} \rangle\bigg].
\end{align*}
To bound these three terms, we will use the following lemma, whose proof is deferred to the end of this section.
\begin{lem}[Uniform high probability bounds for random fluctuation terms]\label{lem:random_fluctuation}
For any $\delta>0$, it holds with probability at least $1-4K^2 n^{-\delta}$ that for any $\vv\in \Gamma_K$ and any distinct pair $(k,l)\in [K]^2$,
\begin{align}
\bigg|\sum_{i\in G^\ast_k}  v_{i}  \langle \vmu_{k}-\vmu_{l}, \vepsilon_{i} \rangle  \bigg| & \leq \|\vmu_k - \vmu_l\| \, \big(n_k+\sqrt{2n_k \log n} + 2 \log n\big)^{1/2}\,\Big( \sum_{i\in G^\ast_k} v_{i}^{2} \Big)^{1/2},\label{Eqn:ineq_a} \\
 \bigg|\sum_{i\in G^\ast_k} v_{i} \,\|\vepsilon_{i}\|^{2}\bigg| & \leq C p^{1/2} [n_{k}^{1/2}+\log^{2}(n)\,]\Big( \sum_{i\in G^\ast_k} v_{i}^{2} \Big)^{1/2},\label{Eqn:ineq_b}\\
\sum_{i\in G^\ast_k}  v_{i} \langle \overline\vepsilon_{l}, \vepsilon_{i} \rangle &\leq C \sqrt{(p+\log{n})\, n_{k} \over n_{l}}\,\Big( \sum_{i\in G^\ast_k} v_{i}^{2} \Big)^{1/2},\label{Eqn:ineq_c}\\
\bigg|\sum_{i\in G^\ast_k}  v_{i} \Big\langle {n_{k}-1 \over n_{k}} \overline\vepsilon_{k\setminus\{i\}},\vepsilon_{i} \Big\rangle \bigg| & \leq C p^{1/2} (\delta \log{n})^{1/2}\Big( \sum_{i\in G^\ast_k} v_{i}^{2} \Big)^{1/2},\label{Eqn:ineq_d}
\end{align}
for some universal constant $C>0$. 
\end{lem}

\vspace{0.5em}
\noindent{\bf Bound $T_{1,kl}$:} By applying the Cauchy-Schwarz inequality and inequality~\eqref{Eqn:ineq_a}, we can bound
\begin{align*}
&|T_{1,kl}(\vv)| =  \Big| \sum_{i\in G^\ast_k} v_{i} \langle \vmu_{k}-\vmu_{l}, \vepsilon_{i} \rangle \sum_{j \in G^\ast_l} v_{j} \langle \vmu_{k}-\vmu_{l}, \vepsilon_{j} \rangle \Big| \\
\leq &  \|\vmu_k - \vmu_l\| ^2\, \Big( \sum_{i\in G^\ast_k} v_{i}^{2} \Big)^{1/2} \, \Big( \sum_{j\in G^\ast_l} v_{j}^{2} \Big)^{1/2} \big(n_k+\sqrt{2n_k \log n} + 2 \log n\big)^{1/2} \,\big(n_l+\sqrt{2n_l \log n} + 2 \log n\big)^{1/2}.
\end{align*}
Throughout the proof, we can always work under the event 
\begin{align}\label{Eqn:event_T}
\{t^{(k,l)} \geq \beta n_{k} n_{l} \|\vmu_{k}-\vmu_{l}\|^{2} / 2 \mbox{ for all distinct pairs } (k,l) \in [K]^{2} \mbox{ and } i \in G_{k}^{*}\},
\end{align}
which according to the choice of $\lambda^\sharp$ in \eqref{Eqn:lambda_def} after Lemma~\ref{lem:separation_bound}, holds with probability at least $1-K^2n^{-\delta} -8n^{-1}$.
Under this event, we get a uniform bound for first sum of $T_{1,kl}$'s in the decomposition~\eqref{Eqn:T_decom} of $T(\vv)$ for all $\vv \in \Gamma_{K}$: 
\begin{align*}
\left| \sum_{k,l=1}^{K} {n_{k} n_{k} \over t^{(k,l)}} T_{1,kl}(\vv) \right|
\leq &\, \frac{2}{\beta} \bigg\{\sum_{k=1}^K \Big( \sum_{i\in G^\ast_k} v_{i}^{2} \Big)^{1/2}  \big(n_k+\sqrt{2n_k \log n} + 2 \log n\big)^{1/2} \bigg\}\\
&\,\cdot \bigg\{\sum_{l=1}^K \Big( \sum_{j\in G^\ast_l} v_{j}^{2} \Big)^{1/2}  \big(n_l+\sqrt{2n_l \log n} + 2 \log n\big)^{1/2} \bigg\}
\\
\overset{(a)}{\leq}&\, {2 \over \beta}\bigg(  \sum_{k=1}^{K}  \sum_{i\in G^\ast_k} v_{i}^{2} \bigg)^{1/2} \cdot \bigg(\sum_{k=1}^K \big(n_k + \sqrt{2n_k \log n} +2\log n\big)\bigg)^{1/2} \\
&\cdot  \bigg( \sum_{l=1}^K \sum_{j\in G^\ast_l} v_{j}^{2} \bigg)^{1/2} \cdot \bigg(\sum_{l=1}^K \big(n_l + \sqrt{2n_l \log n} +2\log n\big)\bigg)^{1/2}.\\
\leq &\, {2 \over \beta} \big(n+\sqrt{2nK\log n} + 2K\log n\big)\,\|\vv\|^2,
\end{align*}
where step (a) is due to the Cauchy-Schwarz inequality, and the last step uses the identity $\sum_{k=1}^K n_k= n$ and inequality $\sum_{k=1}^K\sqrt{n_k} \leq \sqrt{K\sum_{k=1}^K n_k}$.

\vspace{0.5em}
\noindent{\bf Bound $T_{2,kl}$:} Due to the symmetry, we only need to analyze the first sum in $T_{2,kl}$, which can be further decomposed as
\begin{align*}
\sum_{i\in G^\ast_k} v_{i} \langle \overline\vepsilon_{k}-\overline\vepsilon_{l}, \vepsilon_{i} \rangle
&= \sum_{i\in G^\ast_k} v_{i} \Big\langle {1 \over n_{k}} \vepsilon_{i} + {n_{k}-1 \over n_{k}} \overline\vepsilon_{k\setminus\{i\}}-\overline\vepsilon_{l}, \vepsilon_{i}  \Big\rangle \\
&=\sum_{i\in G^\ast_k} \frac{v_{i}}{n_k} \|\vepsilon_j\|^2 -    \sum_{i\in G^\ast_k} v_i \langle \overline\vepsilon_{l}, \vepsilon_{i}  \rangle   +  \sum_{i\in G^\ast_k} v_i \Big\langle {n_{k}-1 \over n_{k}} \overline\vepsilon_{k\setminus\{i\}}, \vepsilon_{i}  \Big\rangle  \\
& =: G_1(\vv) + G_2(\vv) + G_3 (\vv),
\end{align*}
where the three terms $G_1(\vv)$, $G_2(\vv)$ and $G_3(\vv)$ are respectively bounded by using inequalities~\eqref{Eqn:ineq_b}, \eqref{Eqn:ineq_c} and \eqref{Eqn:ineq_d} in Lemma~\ref{lem:random_fluctuation}. Therefore, we can reach
\begin{align*}
\sum_{i\in G^\ast_k} v_{i} \langle \overline\vepsilon_{k}-\overline\vepsilon_{l}, \vepsilon_{i} \rangle
 &\leq C\bigg( \sqrt{\frac{p}{n_k}} + \frac{\log^2(n)\sqrt{p}}{n_k} + \sqrt{\frac{(p+\log n) \log n_k}{n_l}} + \sqrt{\delta p \log n}\bigg) \Big( \sum_{i\in G^\ast_k} v_{i}^{2} \Big)^{1/2}\\
 & \leq C' \bigg(\sqrt{\delta p \log n} + \log^2(n)\sqrt{\frac{p}{\underline n}}\bigg)\Big( \sum_{i\in G^\ast_k} v_{i}^{2} \Big)^{1/2}.
\end{align*}
This implies the following bound on $T_{2,kl}$ due to the symmetry,
\begin{align*}
|T_{2,kl}(\vv)| \leq C'' \bigg(\delta p \log n + \frac{p\log^4(n)}{\underline n}\bigg)\Big( \sum_{i\in G^\ast_k} v_{i}^{2} \Big)^{1/2}\Big( \sum_{j\in G^\ast_l} v_{j}^{2} \Big)^{1/2}.
\end{align*}
Then we may obtain by using the lower bound condition in Lemma~\ref{lem:separation_bound} as $\|\vmu_k-\vmu_l\|^2 \geq C_1 (1-\beta)^{-1}\sqrt{(1+\delta)\,p\log n / m}$ that under the event~\eqref{Eqn:event_T},
\begin{align*}
\left| \sum_{k,l=1}^{K} {n_{k} n_{k} \over t^{(k,l)}} T_{2,kl}(\vv) \right|
\leq &\,\frac{C_2(1-\beta)}{\beta}\Big(\delta \sqrt{m p \log n} + \frac{\sqrt{mp \log^7 n}}{\underline n}\Big) \bigg(\sum_{k=1}^K \Big(\sum_{i\in G^\ast_k} v_{i}^{2} \Big)^{1/2}\bigg)\bigg(\sum_{l=1}^K\Big( \sum_{j\in G^\ast_l} v_{j}^{2} \Big)^{1/2}\bigg)\\
\leq &\, \frac{C_2(1-\beta)K}{\beta}\Big(\delta \sqrt{m p \log n} + \frac{\sqrt{mp \log^7 n}}{\underline n}\Big)  \|\vv\|^2,
\end{align*}
where the last step is due to the Cauchy-Schwarz inequality.

\vspace{0.5em}
\noindent{\bf Bound $T_{3,kl}$:} 
Note that term $|T_{3,kl}(\vv)|$ satisfies
\begin{align*}
|T_{3,kl}(\vv)| \leq &\, \frac{1}{2}\bigg( \sum_{i\in G^\ast_k} v_{i} \langle \vmu_{k}-\vmu_{l}, \vepsilon_{i} \rangle\bigg)^2 + \frac{1}{2}\bigg( \sum_{j\in G^\ast_l} v_{j} \langle \vmu_{k}-\vmu_{l}, \vepsilon_{j} \rangle\bigg)^2 \\
&\, +  \frac{1}{2}\bigg( \sum_{i\in G^\ast_k} v_{i} \langle \overline{\vepsilon}_{k}-\overline{\vepsilon}_{l}, \vepsilon_{i} \rangle\bigg)^2  +   \frac{1}{2}\bigg( \sum_{j\in G^\ast_l} v_{j} \langle \overline{\vepsilon}_{k}-\overline{\vepsilon}_{l}, \vepsilon_{j} \rangle\bigg)^2.
\end{align*}
Therefore, $|T_{3,kl}(\vv)|$ can be bounded by the sum of the upper bounds for $|T_{1,kl}(\vv)|$ and $|T_{2,kl}(\vv)|$.

Putting all pieces together, we can finally reach
\begin{align*}
|T(\vv)| &\leq \bigg|\sum_{k=1}^K \sum_{l\neq k}\frac{n_kn_l}{t^{(k,l)}}\, T_{1,kl}\bigg|  +  \bigg|\sum_{k=1}^K \sum_{l\neq k}\frac{n_kn_l}{t^{(k,l)}}\, T_{2,kl}\bigg|   +   \bigg|\sum_{k=1}^K \sum_{l\neq k}\frac{n_kn_l}{t^{(k,l)}}\, T_{3,kl}\bigg|\\
&\leq \frac{C_3}{\beta} \|\vv\|^2 \Big(n + K\log n + (1-\beta) K \delta \sqrt{m p \log n} + \frac{\sqrt{mp \log^7 n}}{\underline n}\Big).
\end{align*}

\subsection{Proof of Lemma~\ref{lem:random_fluctuation}}

We can apply the Cauchy-Schwarz inequality to obtain
\begin{align*}
&\|\vmu_k-\vmu_l\|^{-1}\bigg|\sum_{i\in G^\ast_k}  v_{i}  \langle \vmu_{k}-\vmu_{l}, \vepsilon_{i} \rangle  \bigg| \leq \Big( \sum_{i\in G^\ast_k} v_{i}^{2} \Big)^{1/2} \Big( \sum_{i\in G^\ast_k} \Big\langle {\vmu_{k}-\vmu_{l} \over \|\vmu_{k}-\vmu_{l}\|}, \vepsilon_{i} \Big\rangle^{2} \Big)^{1/2}.
\end{align*}
Since 
\[
\Big\langle {\vmu_{k}-\vmu_{l} \over \|\vmu_{k}-\vmu_{l}\|}, \vepsilon_{i} \Big\rangle \stackrel{\text{i.i.d.}}{\sim} N(0, 1), \quad i = 1,\dots,n,
\]
we obtain by Lemma~\ref{lem:tail_bound_chisquare} and a union bound argument that with probability at least $1-K^2n^{-1}$,
\[
 \sum_{i\in G^\ast_k} \Big\langle {\vmu_{k}-\vmu_{l} \over \|\vmu_{k}-\vmu_{l}\|}, \vepsilon_{i} \Big\rangle^{2}  \leq n_{k} + \sqrt{2n_{k}\log{n}} + 2\log{n} \quad \mbox{for all } k, l \in [K].
\]
A combination of the preceding two displays yields the first claimed inequality~\eqref{Eqn:ineq_a}.

Since $\sum_{i\in G^\ast_k} v_{i} = 0$ for any $\vv\in\Gamma_K$, we can also write the left hand side of inequality~\eqref{Eqn:ineq_b} as $G_{1}(\vv_k) = \sum_{i\in G^\ast_k} v_{i} (\|\vepsilon_{i}\|^{2}-p)$, which can be viewed as a centered empirical process indexed by $\vv_{k}\in \mathbb R^{n_k}$, the restriction $\vv_{\restriction G^\ast_k}$ of $\vv\in\Gamma_K$ onto $G_k^\ast$. We may assume without loss of generality that $\vv_k\in \bV_k:= \{\vv_{\restriction G^\ast_k} : \vv \in \Gamma_K,\,\|\vv_k\|=1\}$.
By Theorem 4 in~\cite{adamczak2008}, there exists a universal constant $C$ such that for any $t > 0$, 
\[
\Prob \Big( \|G_{1}\|_{\bV_{k}} \geq 2 \E[\|G_{1}\|_{\bV_{k}}] + t \Big) \leq \exp \Big( -{t^{2} \over 3 \tau_{1}^{2}} \Big) + 3 \exp \Big(-{t \over C \|\mathsf{M}_{1}\|_{\psi_{1}}}\Big).
\]
where $\|G_{1}\|_{\bV_{k}} = \sup_{\vv_{k} \in \bV_{k}} |G_{1}(\vv_{k})|$ and $\tau_{1}^{2} = \sup_{\vv_{k} \in \bV_{k}} \sum_{i \in G^{\ast}_{k}} v_{i}^{2} \E[\|\vepsilon_{i}\|^{2}-p]^{2} \leq 2p$, and $\mathsf{M}_{1} = \max_{i \in G^\ast_k} \max_{\vv_{k} \in \bV_{k}} |v_{i} (\|\vepsilon_{i}\|^{2}-p)| \leq \max_{i \in G^\ast_k} |\|\vepsilon_{i}\|^{2}-p|$. By the maximal inequality in Lemma 2.2.2 in~\cite{vandervaartwellner1996} and Lemma~\ref{lem:product_two_gaussian_psi1_norm}, we have 
\[
\|\mathsf{M}_{1}\|_{\psi_{1}} \leq C \log(n_{k}) \max_{i \in G^\ast_k} \|\|\vepsilon_{i}\|^{2}-p\|_{\psi_{1}} \leq C p^{1/2} \log(n_{k}).
\]
By the Cauchy-Schwarz inequality, we have for all $\vv \in \bV_{k}$, 
\[
\Big| \sum_{i\in G^\ast_k} v_{i} (\|\vepsilon_{i}\|^{2}-p) \Big| \leq  \Big( \sum_{i\in G^\ast_k} v_{i} \Big)^{1/2} \Big( \sum_{i\in G^\ast_k} (\|\vepsilon_{i}\|^{2}-p)^{2} \Big)^{1/2} \leq \Big( \sum_{i\in G^\ast_k} (\|\vepsilon_{i}\|^{2}-p)^{2} \Big)^{1/2}.
\]
Then Jensen's inequality implies that 
\[
\E[\|G_{1}\|_{\bV_{k}}] \leq \Big[ \sum_{i\in G^\ast_k} \E(\|\vepsilon_{i}\|^{2}-p)^{2} \Big]^{1/2} = (2n_{k}p)^{1/2}.
\]
Thus with probability at least $1-4n^{-1}$, we have 
\begin{equation}
\label{eqn:G1_bound}
\|G_{1}\|_{\bV_k} \leq C p^{1/2} [n_{k}^{1/2}+\log^{2}(n)],
\end{equation}
which entails the second claimed inequality~\eqref{Eqn:ineq_b}.

Next we prove the third claimed inequality. Note that conditional on $\overline\vepsilon_{l}$, $G_{2}(\vv_{k}) := \langle \overline\vepsilon_{l}, \sum_{i\in G^\ast_k}  v_{i} \vepsilon_{i} \rangle \sim N(0, \|\overline\vepsilon_{l}\|^{2} \sum_{i\in G^\ast_k}  v_{i}^{2})$ is a centered Gaussian process indexed by $\vv_{k} \in \bV_{k}$. By the Borell-Sudakov-Tsirel'son inequality (cf. Theorem 2.5.8 in ~\cite{GineNickl2016}), we have 
\[
\Prob \Big( \|G_{2}\|_{\bV_{k}} \geq \E[\|G_{2}\|_{\bV_{k}} \,|\, \overline{\vepsilon}_{l}] + \tau_{2} \sqrt{2 \log{n}} \,\, \Big|\, \overline{\vepsilon}_{l} \Big) \leq n^{-1},
\]
where $\tau_{2}^{2} = \|\overline{\vepsilon}_{l}\|^{2} \sup_{\vv_{k} \in \bV_{k}} \sum_{i \in G^{\ast}_{k}} v_{i}^{2} \leq  \|\overline{\vepsilon}_{l}\|^{2}$. Then Dudley's entropy integral bound (cf. Corollary 2.2.8 in~\cite{vandervaartwellner1996}) yields that 
\[
\E[\|G_{2}\|_{\bV_{k}} \,|\, \overline{\vepsilon}_{l}] \leq C \|\overline{\vepsilon}_{l}\| n_{k}^{1/2},
\] 
where we have used the fact that the $\varepsilon$-covering entropy of the unit sphere in $\bR^{n_k}$ is at most $Cn_k \log(1/\varepsilon)$ for any $\varepsilon\in(0,1)$.
Combining the last two displays with the inequality 
\[
\Prob (\|\overline{\vepsilon}_{l}\|^{2} \geq n_{l}^{-1} (p + 2\sqrt{p\log{n}} + 2\log{n}) ) \leq n^{-1},
\]
and a union bound argument, we get with probability at least $1-K^2n^{-1}$,
\begin{equation}
\label{eqn:G2_bound}
\|G_{2}\|_{\bV_{k}} \leq C \sqrt{(p+\log{n}) n_{k} \over n_{l}},
\end{equation}
implying the third inequality~\eqref{Eqn:ineq_c}.

Now we prove the last inequality. Note that 
\[
\sum_{i\in G^\ast_k}  v_{i} \Big\langle {n_{k}-1 \over n_{k}} \overline\vepsilon_{k\setminus\{i\}},\vepsilon_{i} \Big\rangle  = {1 \over n_{k}} \sum_{\{(i, j) \in G^\ast_k : i \neq j\}} v_{i} \langle \vepsilon_{i}, \vepsilon_{j} \rangle =: {1 \over n_{k}} U_{1}(\vv_{k}),
\]
where 
\[
U_{1}(\vv_{k}) = \sum_{\{(i, j) \in G^\ast_k : i \neq j\}} {1\over2} (v_{i}+v_{j}) \langle \vepsilon_{i}, \vepsilon_{j} \rangle
\]
is a degenerate $U$-process of order two. To simplify the notation, we may assume $G_k^\ast=\{1,\ldots,n_k\}$ in the rest of this proof.
Applying Lemma~\ref{lem:uniform_hw_quadratic} with 
\[
\cA = \{ A  \, \otimes \, \Id_{p} \,|\, A = \{a_{ij}\}_{i,j \in [n_{k}]}, \, a_{ij} = (v_{i}+v_{j})/2, \, \vv_k \in \bV_{k} \},
\] 
we get 
\[
\Prob\Big( \Big| \|U_{1}\|_{\bV_{k}} - \E[\|U_{1}\|_{\bV_{k}}] \Big| \geq t\Big) \leq 2 \exp \Big[ -C \min \Big( {t^{2} \over \|\vepsilon\|_{\cA}^{2}}, {t \over \sup_{A \in \cA} \|A  \, \otimes \, \Id_{p}\|_{\op}} \Big) \Big]
\]
where $\|U_{1}\|_{\bV_{k}} = \sup_{\vv_k\in \bV_{k}} U_{1}(\vv_{k})$, $\vepsilon^{T} = (\vepsilon_{1}^{T},\dots,\vepsilon_{n_{k}}^{T})$, and 
\[
\|\vepsilon\|_{\cA} = \E \Big[ \sup_{A} \Big\|(A  \, \otimes \, \Id_{p} + A^{T}  \, \otimes \, \Id_{p}) \vepsilon \Big\| \Big].
\] 
By the Cauchy-Schwarz inequality,
\begin{align*}
\|A\|_{\op} = & \max_{\|\vu\|=1} \vu^{T} A \vu = \max_{\|\vu\|=1} \sum_{i,j=1}^{n_{l}} u_{i} u_{j} {v_{i}+v_{j} \over 2} \\
= & \max_{\|\vu\|=1} \Big( \sum_{i=1}^{n_{k}} u_{i} v_{i} \Big) \Big( \sum_{j=1}^{n_{k}} u_{j} \Big) \leq \max_{\|\vu\|=1} \Big( \sum_{i=1}^{n_{k}} u_{i}^{2} \Big)^{1/2} \Big( \sum_{i=1}^{n_{k}} v_{i}^{2} \Big)^{1/2} \Big( \sum_{j=1}^{n_{k}} u_{j}^{2} \Big)^{1/2} n_{k}^{1/2} \leq n_{k}^{1/2}.
\end{align*}
Since $\| A  \, \otimes \, \Id_{p} \|_{\op} = \|A\|_{\op}$, we have  
\[
\sup_{A} \| (A  \, \otimes \, \Id_{p}) \vepsilon \| \leq \sup_{A} \| A  \, \otimes \, \Id_{p} \|_{\op} \|\vepsilon\| \leq n_{l}^{1/2} \|\vepsilon\|.
\]
Then Jensen's inequality yields that 
\[
\|\vepsilon\|_{\cA}^{2} \leq 4 n_{k} \E[\|\vepsilon\|^{2}] = 4 n_{k}^{2} p.
\]
To bound $\E[\|U_{1}\|_{\bV_{k}}]$, we note that 
\begin{align*}
|U_{1}(\vv_{k})| &= \bigg| \sum_{j=1}^{n_{k}} v_{j} \ \Big\langle \vepsilon_{j}, \sum_{i\neq j, i\in [n_k]} \vepsilon_{j} \Big\rangle \bigg|\leq \bigg(\sum_{j=1}^{n_k} v_j^2\bigg)^{1/2} \bigg(\sum_{j=1}^{n_k}\Big\langle \vepsilon_{j}, \sum_{i\neq j, i\in [n_k]} \vepsilon_{j} \Big\rangle^2\bigg)^{1/2}\\
&\leq \bigg(\sum_{j=1}^{n_k}\Big\langle \vepsilon_{j}, \sum_{i\neq j, i\in [n_k]} \vepsilon_{j} \Big\rangle^2\bigg)^{1/2}.
\end{align*}
From Jensen's inequality and the independence between $\varepsilon_j$ and $\sum_{i\neq j, i\in [n_k]} \vepsilon_{j}$, we have
\begin{align*}
\E[\|U_{1}\|_{\bV_{k}}] & \leq \bigg(\sum_{j=1}^{n_k}\E\bigg[\Big\langle \vepsilon_{j}, \sum_{i\neq j, i\in [n_k]} \vepsilon_{j} \Big\rangle^2\bigg]  \bigg)^{1/2} =\bigg(\sum_{j=1}^{n_k}\E\bigg[\Big\|\sum_{i\neq j, i\in [n_k]} \vepsilon_{j} \Big\|^2\bigg]  \bigg)^{1/2} \leq n_k p^{1/2}.
\end{align*}
Thus we see that with probability at least $1-2n^{-\delta}$,
\begin{equation}
\label{eqn:U1_bound}
\|U_{1}\|_{\bV_{k}} \leq C n_{k} p^{1/2} (\delta \log{n})^{1/2},
\end{equation}
which implies the last claimed inequality~\eqref{Eqn:ineq_d}.

\section{Supporting lemmas}
\label{sec:supporting_lemmas}

\begin{lem}[Tail bound for $\chi^{2}$ distributions]
\label{lem:tail_bound_chisquare}
If $\vZ \sim N(\vzero, \Id_{p})$, then for all $t > 0$,
\begin{align*}
\Prob(\|\vZ\|^{2} \geq p + 2\sqrt{pt} + 2t) & \leq e^{-t}, \\
\Prob(\|\vZ\|^{2} \leq p - 2\sqrt{pt}) &\leq e^{-t}.
\end{align*}
\end{lem}
\begin{proof}[Proof of Lemma~\ref{lem:tail_bound_chisquare}]
See Lemma 1 in~\cite{LaurentMassart2000_AoS}.
\end{proof}

\begin{lem}[Deviation of Gaussian random matrices]
\label{lem:bound_Sv}
If $\mathcal E\in\mathbb R^{p\times n}$ has i.i.d. $N(0,\sigma^2)$ entries, then 
\begin{align*}
\mathbb P\big( \|\mathcal E\|_{\op} \geq \sigma(\sqrt{n}+\sqrt{p} +\sqrt{2t}\,)\big) \leq e^{-t},\quad\forall t>0.
\end{align*}
\end{lem}
\begin{proof}[Proof of Lemma~\ref{lem:bound_Sv}]
See Corollary 5.35 in~\cite{vershynin_2012}.
\end{proof}

\begin{lem}
\label{lem:product_two_gaussian_psi1_norm}
Let $\vepsilon_{1},\vepsilon_{2}$ be i.i.d. $N(\vzero, \Id_{p})$. Then there exists a universal constant $C$ such that 
\[
\| \| \vepsilon_{1} \|^{2} - p \|_{\psi_{1}} + \| \langle \vepsilon_{1},\vepsilon_{2} \rangle \|_{\psi_{1}} \leq C p^{1/2}.
\]
\end{lem}
\begin{proof}[Proof of Lemma~\ref{lem:product_two_gaussian_psi1_norm}]
Note that $\langle \vepsilon_{1}, \vepsilon_{2} \rangle=\sum_{j=1}^p \varepsilon_{1j}  \varepsilon_{2j}$, and each additive component $\varepsilon_{1j}  \varepsilon_{2j}$ is sub-exponential with $\|\varepsilon_{1j}  \varepsilon_{2j}\|_{\psi_{1}}\leq \|\varepsilon_{1j}\|_{\psi_{2}}  \|\varepsilon_{2j}\|_{\psi_{2}}=1$ (cf. Lemma 2.7.7 in~\cite{Vershynin2018_Cambridge}).
By Bernstein's inequality (cf. Theorem 2.8.2 in~\cite{Vershynin2018_Cambridge}), there exists a universal constant $C_{1}$ such that for any $t > 0$, 
\[
\Prob(|\langle \vepsilon_{1}, \vepsilon_{2} \rangle| \geq t) \leq 2 \exp [ -C_{1} \min({t^{2} / p}, t) ].
\]
Let $C$ be a large positive real number. By integration-by-parts and change-of-variables, we have 
\begin{align*}
\E\Big[ \exp\Big({|\langle \vepsilon_{1}, \vepsilon_{2} \rangle| \over C}\Big) \Big] = & \int_{1}^{\infty} \Prob \Big( \exp\Big({|\langle \vepsilon_{1}, \vepsilon_{2} \rangle| \over C}\Big) > t \Big) \, \rd t \\
= & \int_{1}^{\infty} \Prob \Big( |\langle \vepsilon_{1}, \vepsilon_{2} \rangle| > C \log{t} \Big) \, \rd t \\
= & \int_{0}^{\infty} \Prob \Big( |\langle \vepsilon_{1}, \vepsilon_{2} \rangle| > C x \Big) \re^{x} \, \rd x \\
\leq & 2 \int_{0}^{p/C} \re^{-{C_{1} C^{2} \over p} x^{2} + x} \, \rd x + 2 \int_{p/C}^{\infty} \re^{-(C_{1}C-1)x} \, \rd x \\
\leq & 2 \re^{p \over 4 C_{1}C^{2}} \sqrt{\pi p \over C_{1}C^{2}} + {2 \over C_{1}C-2} \re^{-(C_{1}C-2){p \over C}}.
\end{align*}
Thus if we take $C = K p^{1/2}$ for some large enough universal constant $K > 0$, then 
\[
\E\Big[ \exp\Big({|\langle \vepsilon_{1}, \vepsilon_{2} \rangle| \over C}\Big) \Big] \leq 2,
\]
which implies that $\| \langle \vepsilon_{1},\vepsilon_{2} \rangle \|_{\psi_{1}} \leq K p^{1/2}$. The $\psi_{1}$ norm bound for $\| \vepsilon_{1} \|^{2} - p$ follows from similar lines. 
\end{proof}

\begin{lem}[Uniform Hanson-Wright inequality for Gaussian quadratic forms]
\label{lem:uniform_hw_quadratic}
Let $\vepsilon \sim N(0, \Id_{p})$ and $\cA$ be a bounded class of $p \times p$ matrices. Consider the random variable 
\[
Z = \sup_{A \in \cA} (\vepsilon^{T} A \vepsilon - \E[\vepsilon^{T} A \vepsilon]).
\]
Then there exists a universal constant $C$ such that for any $t > 0$, 
\[
\Prob(|Z - \E[Z]| \geq t) \leq 2 \exp\Big[ -C \min \Big( {t^{2} \over \|\vepsilon\|_{\cA}^{2}}, {t \over \sup_{A \in \cA} \|A\|_{\op}} \Big) \Big],
\]
where $\|\vepsilon\|_{\cA} = \E[\sup_{A \in \cA} \|(A+A^{T}) \vepsilon\|]$.
\end{lem}
\begin{proof}[Proof of Lemma~\ref{lem:uniform_hw_quadratic}]
Note that the standard Gaussian random vector $\vepsilon$ satisfies the concentration inequality 
\[
\Prob (|\varphi(\vepsilon) - \E[\varphi(\vepsilon)]| \geq t) \leq 2 \exp (- {t^{2} / 2} )
\]
for any $t > 0$ and every 1-Lipschitz function $\varphi : \bR^{p} \to \bR$ such that $\E[|\varphi(\vepsilon)|] < \infty$ (cf. Theorem 2.5.7 in~\cite{GineNickl2016}). Then the lemma follows from Theorem 2.10 in~\cite{adamczak2015}.
\end{proof}

\section*{Acknowledgement}
The authors would like to thank two anonymous referees and the Associate Editor Prof. Stephane Boucheron for their many constructive comments. X. Chen's research was supported in part by NSF CAREER Award DMS-1752614, UIUC Research Board Award RB18099, and a Simons Fellowship. Y. Yang's research was supported in part by NSF DMS-1810831. X. Chen acknowledges that part of this work was carried out in the Institute for Data, System, and Society (IDSS) at Massachusetts Institute of Technology.

\bibliographystyle{plain}
\bibliography{sharp_threshold_mixture}

\end{document}